\documentclass{amsart}[12 pt]

\usepackage{latexsym}

\usepackage{amsmath,amssymb,dsfont,amsfonts}
\usepackage{mathrsfs}
\usepackage{verbatim}
\usepackage{graphicx}
\usepackage{graphics}
\usepackage{geometry}
\usepackage{booktabs}
\usepackage{enumitem}

\newtheorem{proposition}{Proposition}[section]
\newtheorem{theorem}{Theorem}[section]
\newtheorem{definition}{Definition}[section]
\newtheorem{corollary}{Corollary}[section]
\newtheorem{lemma}{Lemma}[section]
\newtheorem{remark}{Remark}[section]
\newtheorem{example}{Example}[section]

\numberwithin{equation}{section}

\allowdisplaybreaks

\begin{document}
\markboth{}{}

\title[Weak differentiability of Wiener functionals
and occupation times] {Weak differentiability of Wiener functionals
and occupation times}

\author{Dorival Le\~ao}

\address{Departamento de Matem\'atica Aplicada e Estat\'istica. Universidade de S\~ao
Paulo, 13560-970, S\~ao Carlos - SP, Brazil} \email{estatcamp@icmc.usp.br}

\author{Alberto Ohashi}

\address{Departamento de Matem\'atica, Universidade Federal da Para\'iba, 13560-970, Jo\~ao Pessoa - Para\'iba, Brazil}\email{alberto.ohashi@pq.cnpq.br; amfohashi@gmail.com}

\author{Alexandre B. Simas}

\address{Departamento de Matem\'atica, Universidade Federal da Para\'iba, 13560-970, Jo\~ao Pessoa - Para\'iba, Brazil}\email{alexandre@mat.ufpb.br}

\thanks{Corresponding author: Alberto Ohashi}
\date{\today}


\begin{center}

\end{center}

\begin{abstract}
In this paper, we establish a universal variational characterization of the non-martingale components associated with weakly differentiable Wiener functionals in the sense of Le\~ao, Ohashi and Simas. It is shown that any Dirichlet process (in particular semimartingales) is a differential form w.r.t Brownian motion driving noise. The drift components are characterized in terms of limits of integral functionals of horizontal-type perturbations and first-order variation driven by a two-parameter occupation time process. Applications to a class of path-dependent rough transformations of Brownian paths under finite $p$-variation ($p\ge 2$) regularity is also discussed. Under stronger regularity conditions in the sense of finite $(p,q)$-variation, the connection between weak differentiability and two-parameter local time integrals in the sense of Young is established.
\end{abstract}

\maketitle

\section{Introduction}
Recently, a new branch of stochastic calculus has appeared, known as functional
It\^o calculus, which results to be an extension of classical It\^o calculus to non-anticipative functionals depending on the whole path of a noise $W$ and not only on its current value. Inspired by applications in mathematical finance, Dupire~\cite{dupire} introduced the so-called horizontal, vertical and second-order vertical derivative operators which describe semimartingales of the form $F(W)$ for a given smooth non-anticipative functional $F$. Following the insightful Dupire's idea, many authors have investigated It\^o-type formulas for adapted processes of the form $F(W)$ under a variety of type of noises $W$ and regularity conditions on $F$. See e.g Cont and Fourni\'e \cite{cont,cont1}, Ananova and Cont \cite{ananova}, Keller and Zhang \cite{keller}, Viens and Zhang \cite{viens}, Cosso and Russo \cite{cosso1,cosso2}, Oberhauser \cite{ober}, Ohashi, Shamarova and Shamarov \cite{ohashi}, Saporito \cite{saporito} and Resestolato \cite{resestolato1}. Cont \cite{cont2} and Peng and Song \cite{peng2} proposed Sobolev-type approaches for It\^o processes. Towards a deeper understanding of the non-Markovian phenomenon, the search for a sound definition of path-dependent PDEs has attracted great interest by having functional stochastic calculus as the starting point. In this direction, see e.g Peng and Wang \cite{peng1}, Ekren, Keller, Touzi and Zhang \cite{touzi1}, Ekren, Touzi and Zhang \cite{touzi2,touzi3}, Ekren and Zhang \cite{ekren},  Cosso and Russo \cite{cosso3}, Barrasso and Russo \cite{barrasso}, Bion-Nadal \cite{nadal}, Flandoli and Zanco \cite{flandoli}, Cosso, Federico, Gozzi, Resestolato and Touzi \cite{resestolato} and Buckdahn, Keller, Ma and Zhang \cite{buckdhan1}. Applications to Finance and stochastic control are considered by e.g Jazaerli and Saporito \cite{jazaerli}, Possama\"i, Tan and Zhou~\cite{possamai}, Cont and Lu \cite{cont3}, Pham and Zhang \cite{pham}. See also the recent monograph of Zhang \cite{zhang} for further references in this topic.

Inspired by the discretization scheme introduced by Le\~ao and Ohashi~\cite{LEAO_OHASHI2013}, Le\~ao, Ohashi and Simas \cite{LOS} have developed a non-anticipative differential calculus over the space of square integrable Wiener functionals which proves to be a weaker notion of the functional It\^o calculus studied by previous works \cite{dupire,cont,cont1,cont2,peng2,cosso1,ober,ohashi}. The building block of the theory is composed by what we call a \textit{stable imbedded discrete structure} $\mathcal{Y}  = \big((X^k)_{k\ge 1},\mathscr{D}\big)$ (see Definition \ref{stabledef}) equipped with a continuous-time random walk approximation $\mathscr{D} = \{\mathcal{T},A^{k}; k\ge 1\}$ driven by a suitable class of waiting times $\mathcal{T} = \{T^{k}_n; n,k\ge 1\}$. The class $\mathcal{T}$ describes the evolution of the Brownian motion at small scales and $(X^k)_{k\ge 1}$ is a suitable type of approximation (adapted to $A^k$) for a given target process $X$. In \cite{LOS}, the authors show that a large class of Wiener functionals admits a stable imbedded discrete structure which has to be interpreted as a rather weak concept of continuity w.r.t Brownian motion driving noise. One major advantage of this theory is the possibility of computing the sensitivities of $X$ w.r.t $B$ via simple and explicit differential-type operators based on $\mathscr{D}$. For a concrete application of this theory, we refer the reader to \cite{LEAO_OHASHI2017.1,LEAO_OHASHI2017.2,bezerra}.

Based on limits of generic imbedded discrete structures $\mathcal{Y}  = \big((X^k)_{k\ge 1},\mathscr{D}\big)$, it is revealed a large class of square-integrable Wiener functionals can be decomposed into

\begin{equation}\label{wdrepint}
X(t) = X(0) + \int_0^t \mathcal{D}X(s)dB(s) + V_X(t); 0\le t\le T,
\end{equation}
where $dB$ is the It\^o integral w.r.t a Brownian motion $B$ and $V_X$ is a continuous process which carries all possible orthogonal variations (in a sense made precise in \cite{LOS}) of $X$ w.r.t $B$. This class of Wiener functionals are called \textit{weakly differentiable} (see Definition \ref{defweakder}) which turns out to encompass semimartingales and other classes of processes with unbounded variation drifts. The non-anticipative process $\mathcal{D}X$ is constructed based on limits of the variations

$$\mathcal{D}^{\mathcal{Y},k}X (T^{k}_n) = \frac{\Delta X^k(T^{k}_n)}{\Delta A^{k}(T^{k}_n)}; n\ge 1,$$
for a given stable imbedded discrete structure $\mathcal{Y}  = \big((X^k)_{k\ge 1},\mathscr{D}\big)$.

The goal of this article is to present a universal variational characterization of the drift $V_X$ for a given weakly differentiable Wiener functional $X$ of the form (\ref{wdrepint}). We are interested in establishing a common and universal probabilistic structure of $V_X$ which unifies all types of orthogonal variations of $X$ w.r.t the Brownian motion driving noise. Our first result in this direction (see Proposition \ref{ulocaltime}) characterizes the drift

\begin{equation}\label{vintr}
V_X(\cdot)=\lim_{k\rightarrow \infty}\Bigg[\int_0^\cdot \mathbb{D}^{\mathcal{Y},k,h}X(s)ds - \frac{1}{2}\int_0^\cdot\int_{-\infty}^{+\infty}\nabla^ {\mathcal{Y},k,v}X(s,x)d_{(s,x)}\mathbb{L}^{k,x}(s)\Bigg]
\end{equation}
weakly in $\mathbf{B}^2(\mathbb{F})$ (see (\ref{B2norm}) for the definition of this space) for every stable imbedded discrete structure $\mathcal{Y} = \big((X^k)_{k\ge1},\mathscr{D}\big)$, where $\mathbb{D}^{\mathcal{Y},k,h}X$ and $\nabla^ {\mathcal{Y},k,v}X$ encode an horizontal-type variation (see (\ref{ykh})) and a first-order variation (see (\ref{fkdd})) of $X$ w.r.t $B$, respectively. The integrator process in the double integral $d_{(s,x)}$ in (\ref{vintr}) is given by a suitable two-parameter occupation-type process $(s,x)\mapsto\mathbb{L}^{k,x}(s)$ for $\mathscr{D}$ given by (\ref{Lk}).

Our characterization of the drift component of weakly differentiable processes in terms of occupation times is reminiscent from the classical results of Tanaka-Wang-Meyer where the absence of the pointwise spatial second-order derivative is mitigated by means of local-times via the occupation time formula and mean value theorem. In the path-dependent case, the absence of second-order vertical derivative in the sense of Dupire can also be compensated by means of local-times as demonstrated by \cite{ohashi,saporito} at the expense of several pathwise regularity conditions. In this article, we show that this phenomenon is universal in the sense that drift components of \textit{every} weakly differentiable Wiener functional is totally characterized by (\ref{vintr}). As a by-product, Proposition \ref{ulocaltime} provides a complete variational characterization of Dirichlet processes (see e.g \cite{bertoin}) (in particular square-integrable semimartingales) in full generality with none pathwise regularity condition.

In the second part of this paper, we study the relation between weak differentiability and finite $(p,q)$-variation for $1\le p,q < \infty$ in time and space. In particular, it is shown that the weak differentiability concept also covers Wiener functionals with drifts admitting rather weak path-regularity beyond Dirichlet processes. Let $\mathcal{W}_p(a,b)$ be the space of real-valued functions $f:[a,b]\rightarrow\mathbb{R}$ such that

$$\|f\|^p_{[a,b];p}:= \sup_{\Pi}\sum_{x_i\in \Pi}|f(x_i)-f(x_{i-1})|^p <\infty$$
where $p\ge 1$ and the $\sup$ is taken over all partitions $\Pi$ of a compact set $[a,b]\subset \mathbb{R}$. Theorem \ref{roughcasePROP} reveals that if a square-integrable Wiener functional $X$ has the form

$$X(t) = \int_0^t Z(s)dB(s) + \int_0^t Y(s)dg(s)$$
where $Z$ satisfies some regularity conditions (see section \ref{roughcaseSUB}), $Y \in \mathcal{W}_p(0,T)$ a.s and $g\in \mathcal{W}_q(0,T)$ such that $1/p + 1/q>1$ and $2 \le q  < \infty$, then $X$ is weakly differentiable and

$$\int_0^\cdot Y(s)dg(s) = \lim_{k\rightarrow \infty}\Bigg[\int_0^\cdot \mathbb{D}^{\mathcal{Y},k,h}X(s)ds - \frac{1}{2}\int_0^\cdot\int_{-\infty}^{+\infty}\nabla^ {\mathcal{Y},k,v}X(s,x)d_{(s,x)}\mathbb{L}^{k,x}(s)\Bigg]$$
weakly in $\mathbf{B}^2(\mathbb{F})$ for every stable imbedded discrete structure $\mathcal{Y}$ associated with $X$. Under a suitable space-time regularity condition (in the sense of Young~\cite{young1}) on the gradient functional $\nabla_x$, Theorem \ref{youngTh} identifies the associated drift (\ref{vintr}) of a given $X=F(B)$ in terms of

\begin{equation}\label{eqint}
\int_0^t \mathcal{D}^h F_s(B_s) ds  -\frac{1}{2}\int_0^\cdot\int_{-\infty}^{+\infty}\nabla_x F_s(\textbf{t}(B_s,x))d_{(s,x)}\mathcal{\ell}^{x}(s)
\end{equation}
where $\mathcal{D}^h$ is similar in nature to the horizontal derivative operator of \cite{cont}, $\textbf{t}(B_s,x)$ is a ''terminal value modification'' defined by the following pathwise operation
$$
\textbf{t}(B_s,x):=\left\{
\begin{array}{rl}
B(u); & \hbox{if}~0\le u<s \\
x;& \hbox{if}~u=s
\end{array}
\right.
$$
and the $d_{(s,x)}\ell^x(s)$-integral in (\ref{eqint}) is the pathwise 2D Young integral w.r.t the Brownian local-time $\{\ell^x(s); 0\le s\le T,x\in\mathbb{R}\}$ random field. 

Finally, it is worth to mention that the decompositions found in this article do not presuppose pathwise regularity conditions on a given functional representation but only the existence of a stable imbedded discrete structure $\mathcal{Y}  = \big((X^k)_{k\ge 1},\mathscr{D}\big)$. This might be interpreted as a rather weak concept of continuity w.r.t Brownian motion driving noise which is conveniently found in a large class of Wiener functionals like Dirichlet processes (see Proposition \ref{ulocaltime} and Theorem 4.2 in \cite{LOS}) and even for processes with irregular drifts with very high variation (see section \ref{roughcaseSUB}). In this direction, we drive the attention to Oberhauser \cite{ober} who develops nonlinear It\^o-type decompositions based on a weak concept of continuity. In \cite{ober}, continuity requires stability under piecewise constant noise approximations in probability w.r.t the driving noise and running time local perturbations for a given functional representation. In contrast to \cite{ober}, we do not need to work directly with a functional representation but rather with stochastic approximations which can be chosen case by case. Then, we replace continuity of a given functional representation w.r.t driving noise by the existence of imbedded structures. In addition, our differential operators are computed via a finite-dimensional probabilistic approximation procedure, while \cite{ober} assumes pathwise differentiability conditions in the sense of Dupire subject to continuity along piecewise constant approximations. By exploring the fine structure of the Brownian motion driving noise, we are also able to go beyond It\^o-type processes as considered in \cite{ober}. However, we are still restricted to a fixed probability measure. For partial results in the non-linear case, we refer the reader to \cite{LEAO_OHASHI2017.2}.




The remainder of this article is organized as follows. The next section fixes notation and summarizes some basic results of \cite{LOS}. Section \ref{LOCALTIMESEC} presents the variational characterization of drifts of weakly differentiable processes. Section \ref{pqregularitysection} presents a class of path-dependent rough transformations of Brownian paths. Section \ref{Youngsection} analyzes the convergence of the occupation time integrals in (\ref{vintr}) to a 2D-Young integrals under stronger regularity than previous version.

\section{Preliminaries}\label{preli}
Let $(\Omega, \mathcal{F}, \mathbb{F}, \mathbb{P})$ be a stochastic basis equipped with a one-dimensional Brownian motion $B$ where $\mathbb{F} = (\mathcal{F}_t)_{t\ge 0}$ is the $\mathbb{P}$-augmented filtration generated by the Brownian motion. For a given terminal time $0< T< \infty$, let $\mathbf{B}^p(\mathbb{F})$ be the Banach space of all $\mathbb{F}$-adapted real-valued c\`adl\`ag processes $X$ equipped with the norm 

\begin{equation}\label{B2norm}
\|X\|_{\mathbf{B}^p}:=\big(\mathbb{E}|X^*|^p\big)^{\frac{1}{p}},\quad \text{where}~X^*:=\sup_{0\le t\le T}|X(t)|,
\end{equation}
for $1\le p < \infty$. If $Y$ is a real-valued c\`adl\`ag process, then we denote $\Delta Y(t):=Y(t) - Y(t-)$ where $Y(t-)$ is the left-hand limit at time $t$. The elements of $\mathbf{B}^2(\mathbb{F})$ will be called square-integrable Wiener functionals. Throughout this article, we make use of the following standard notation: If $Y$ is a square-integrable semimartingale w.r.t a given filtration, then usual quadratic variation is denoted by $[Y,Y]$. We denote $L^2_a(\mathbb{P}\times Leb)$ as the usual subspace of real-valued $\mathbb{F}$-adapted processes such that

$$\mathbb{E}\int_0^T|X(t)|^2dt< \infty.$$
Let us denote $\mathbf{H}^2(\mathbb{F})$ as the space of square-integrable $\mathbb{F}$-martingales starting at zero. In order to make this paper self-contained, we recall the fundamental variational operators which describe the so-called weakly differentiable Wiener functionals as described in \cite{LOS}. The weak functional stochastic calculus developed in \cite{LOS} is constructed from a class of pure jump processes driven by suitable waiting times which describe the local behavior of the Brownian motion: For a given sequence $(\epsilon_k)_{k\ge 1}$ of positive numbers such that $\sum_{k\ge 1}\epsilon_k^2 < \infty$, we define $\mathcal{T}:=\{T^k_n; n\ge 0, k\ge 1\}$ as follows: We set $T^{k}_0:=0$ and

$$
T^{k}_n := \inf\{T^{k}_{n-1}< t <\infty;  |B(t) - B(T^{k}_{n-1})| = \epsilon_k\}, \quad n \ge 1.
$$
Then, we define

$$
A^{k} (t) := \sum_{n=1}^{\infty} \epsilon_k \sigma^{k}_n1\!\!1_{\{T^{k}_n\leq t \}};~t\ge0,
$$
where the size of the jumps $\{\sigma^{k}_n; n\ge 1\}$ is given by

$$
\sigma^{k}_n:=\left\{
\begin{array}{rl}
1; & \hbox{if} \ \Delta A^{k}(T^{k}_n) > 0 \\
-1;& \hbox{if} \ \Delta A^{k}(T^{k}_n)< 0. \\
\end{array}
\right.
$$
One can easily check that $\{\sigma^{k}_n; n\ge 1\}$ is an i.i.d sequence of $\frac{1}{2}$-Bernoulli random variables for each $k\ge 1$. Moreover, $\{\Delta A^{k}(T^{k}_n); n\ge 1\}$ is independent from $\{\Delta T^{k}_n; n\ge 1\}$ for each $k\ge 1$. By construction

$$
\sup_{t\ge 0}|A^{k}(t) - B(t)|\le \epsilon_k~a.s
$$
for every $k\ge 1$. Let $\mathbb{F}^{k} := \{ \mathcal{F}^{k}_t; t\ge 0 \} $ be the natural filtration generated by $\{A^{k}(t);  t \ge 0\}$. Let $\mathcal{F}^{k}_\infty$ be the completion of $\sigma(A^{k}(s); s\ge 0)$ and let $\mathcal{N}_{k}$ be the $\sigma$-algebra generated by all $\mathbb{P}$-null sets in $\mathcal{F}^{k}_\infty$. With a slight abuse of notation, we write $\mathbb{F}^{k} = (\mathcal{F}^{k}_t)_{t\ge 0}$, where $\mathcal{F}^{k}_t$ is the usual $\mathbb{P}$-augmentation (based on $\mathcal{N}_k$) satisfying the usual conditions. Of course, $\mathbb{F}^k\subset \mathbb{F}; k\ge 1$. One can check that $A^k$ is a square-integrable $\mathbb{F}^k$-martingale (see Lemma 2.1 in \cite{LOS}). Let us also denote 

\begin{equation}\label{numberhitting}
N^k(t) := \max\{n; T^k_n\le t\}; 0\le t\le T.
\end{equation}

\begin{definition}
The structure $\mathscr{D} = \{\mathcal{T}, A^{k}; k\ge 1\}$ is called a \textbf{discrete-type skeleton} for the Brownian motion.
\end{definition}
For a given choice of discrete-type skeleton $\mathscr{D}$, \cite{LOS} constructs a differential theory based on functionals written on $\mathscr{D}$. In the sequel, we present some of the basic objects developed in \cite{LOS}.

\begin{definition}
A \textbf{Wiener functional} is an $\mathbb{F}$-adapted continuous process which belongs to $\mathbf{B}^2(\mathbb{F})$.
\end{definition}

\begin{definition}\label{GASdef}
For a given discrete-type skeleton $\mathscr{D}$, we say that a pure jump $\mathbb{F}^k$-adapted process of the form

$$
X^k(t) = \sum_{n=0}^\infty X^{k}(T^k_n)1\!\!1_{\{T^k_n\le t < T^k_{n+1}\}}; 0\le t\le T,
$$
is a \textbf{good approximating sequence} (henceforth abbreviated by GAS) w.r.t $X$ if $\mathbb{E}[X^k, X^k|(T)<\infty$ for every $k\ge 1$ and

$$\lim_{k\rightarrow+\infty}X^k=X \quad \text{weakly in}~\mathbf{B}^2(\mathbb{F}).$$
\end{definition}

\begin{definition}\label{IDS}
An \textbf{imbedded discrete structure} $\mathcal{Y} = \big( (X^k)_{k\ge 1}, \mathscr{D}\big)$ for a Wiener functional $X$ consists of the following elements:
\begin{itemize}
  \item A discrete-type skeleton $\mathscr{D}=\{\mathcal{T}, A^{k}; k\ge 1\}$ for the Brownian state $B$.
  \item A GAS $\{X^k; k\ge 1\}$ w.r.t $X$ associated with the above discrete-type skeleton.
\end{itemize}
\end{definition}
\begin{remark}
We stress imbedded discrete structures exist for every continuous process in $\mathbf{B}^2(\mathbb{F})$. For more details, see Lemma 3.1 in \cite{LOS}.
\end{remark}

For a given Wiener functional $X$, let us choose an imbedded discrete structure $\mathcal{Y} = \big( (X^k)_{k\ge 1}, \mathscr{D}\big)$ associated with $X$. In order to characterize the infinitesimal variations of $X$ w.r.t $B$, the first step is to understand what happens at the level of a given structure $\mathcal{Y}$ and this will be achieved by means of suitable differential operators $(\mathbb{D}^{\mathcal{Y},k}X, \mathbb{U}^{\mathcal{Y},k} X)$ as introduced in \cite{LOS}. In order to define these basic operators properly, we make use of the so-called dual predictable projections commonly used in the classical French school of general theory of processes, which we summarize here for completeness. For further details, we refer the reader to the references \cite{he,dellacherie}.

For a given $\mathbb{F}^k$-adapted locally integrable increasing process $Y$, let $\mu_{Y}$ be the Dol\'eans measure generated by $Y$, i.e.,

$$\mu_{Y}(H):=\mathbb{E}\int_0^T\mathds{1}_{H}(s)dY(s); H\in \mathcal{F}^k_T\times \mathcal{B}([0,T]),$$
where $\mathcal{B}([0,T])$ is the Borel sigma algebra of $[0,T]$. In this case, we say $\mu_Y$ is generated by $Y$. In the sequel, $\mathcal{P}^k$ is the $\mathbb{F}^k$-predictable sigma-algebra of $[0,T]\times \Omega$. For a given Dol\'eans measure $\mu_Y$, we can associate a measure not charging evanescent sets as follows: For each non-negative bounded measurable process $Q$, we set

$$\mu^{k,p}_Y(Q):=\mu_Y(^{p,k}Q)$$
where $^{p,k}Q$ is the $\mathbb{F}^k$-predictable projection of $Q$ (see e.g Th 5.12 in \cite{he}). It is known (see e.g Th. 5.20 in \cite{he}), there exists a unique $\mathbb{F}^k$-predictable increasing process $Y^{p,k}$ such that $\mu^{k,p}_Y$ is generated by $Y^{p,k}$ and we call $Y^{p,k}$ as the $\mathbb{F}^k$-dual predictable projection of $Y$. This procedure clearly extends to any $\mathbb{F}^k$-adapted process $Y$ with locally integrable variation. Having said that, for a structure $\mathcal{Y} = \big((X^k)_{k\ge 1},\mathscr{D}\big)$, we then set

$$
\mathcal{D}^{\mathcal{Y},k}X(u) :=  \sum_{\ell=1}^{\infty} \frac{\Delta X^{k} (T^{k}_\ell)}{\Delta A^{k}(T^{k}_\ell)} 1\!\!1_{\{T^{k}_\ell=u\}}
$$
and
\begin{equation}\label{weakinfG}
U^{\mathcal{Y},k}X(u):=\mathbb{E}_{\mu_{[A^{k}]}}\Bigg[\frac{\mathcal{D}^{\mathcal{Y},k}X}{\Delta A^{k}}\Big|\mathcal{P}^k\Bigg](u);~
\end{equation}
for $0\le u\le T$ and $k\ge 1$, where we use the shorthand notation $[A^k] = [A^k,A^k]$. The process (\ref{weakinfG}) is the unique (up to sets of $\mu_{[A^{k}]}$-measure zero) $\mathbb{F}^k$-predictable process which realizes

$$\Big(\int_0^\cdot\frac{\mathcal{D}^{\mathcal{Y},k}X}{\Delta A^{k}}d [A^{k},A^{k}]\Big)^{p,k} = \int_0^\cdot \mathbb{E}_{\mu_{[A^{k}]}}\Big[\frac{\mathcal{D}^{\mathcal{Y},k}X}{\Delta A^{k}}\big|\mathcal{P}^k\Big] d\langle A^{k},A^{k}\rangle,$$
where it is understood that the stochastic process $\mathcal{D}^{\mathcal{Y},k}X/\Delta A^{k}$ is null on the complement of the union of stochastic intervals $\cup_{n=1}^\infty [[T^{k}_n,T^{k}_n]]$, where $[[T^k_n,T^k_n]]: = \{(t,\omega); T^k_n(\omega)=t\};n\ge 1$. Here, $\langle A^k,A^k\rangle$ is the angle bracket of the martingale $A^k$. See Section 3 in \cite{LOS} for more details. In addition, let us denote

\begin{eqnarray}
\label{intf}\oint_0^t\mathcal{D}^{\mathcal{Y},k}X(s)dA^{k} (s)&:=&\int_0^t\mathcal{D}^{\mathcal{Y},k}X(s) dA^{k} (s)\\
\nonumber& &\\
\nonumber&-&\Big(\int_0^{\cdot}\mathcal{D}^{\mathcal{Y},k} X(s) d A^{k} (s)\Big)^{p,k}(t),
\end{eqnarray}
where $\oint$ is the $\mathbb{F}^k$-optional integral as introduced by Dellacherie and Meyer (see Chap 8, section 2 in \cite{dellacherie}) for optional integrands and $dA^k$ in the right-hand side of (\ref{intf}) is interpreted in the Lebesgue-Stieljtes sense. Proposition 3.1 in \cite{LOS} allows us to decompose a given structure $\mathcal{Y}$ as follows

\begin{equation}\label{discretediffform}
X^{k}(t) = X^k(0) + \oint_0^t \mathbb{D}^{\mathcal{Y},k}X(s)dA^{k}(s) +\int_0^t\mathbb{U}^{\mathcal{Y},k} X (s) ds,~0\le t \le T,
\end{equation}
where
\begin{equation}\label{uncOPERATORS}
\mathbb{D}^{\mathcal{Y},k} X(t) :=  \sum_{\ell=1}^{\infty} \mathcal{D}^{\mathcal{Y},k}X(t) 1\!\!1_{\{T^{k}_\ell\le t<  T^{k}_{\ell+1}\}},\quad \mathbb{U}^{\mathcal{Y},k}X(t) := U^{\mathcal{Y},k}X(t) \frac{d\langle A^{k}, A^{k}\rangle}{dt}
\end{equation}
for $0\le t\le T$. The differential form (\ref{discretediffform}) is the starting point to analyse limiting differential forms describing Wiener functionals. In order to study the asymptotic properties of the variational operators, it is necessary to impose the following compactness assumption:

\begin{definition}\label{fenergy}
Let $\mathcal{Y}=\big( (X^k)_{k\ge 1}; \mathscr{D}\big)$ be an imbedded discrete structure for a Wiener functional $X$. We say that $\mathcal{Y}$ has \textbf{finite energy} if
$$\mathcal{E}^{2,\mathcal{Y}}(X):=\sup_{k\ge 1}\mathbb{E}\sum_{n\ge 1}|\Delta X^{k}(T^k_n)|^21\!\!1_{\{T_{n}^k \leq T\}} < \infty.$$
\end{definition}
In addition to compactness, we need a more refined notion of structures associated with Wiener functionals as follows. From (\ref{discretediffform}), we know that each imbedded discrete structure $\mathcal{Y}=\big( (X^k)_{k\ge 1}; \mathscr{D}\big)$ carries a sequence of $\mathbb{F}^k$-special semimartinagle decompositions

$$X^k(t) = X^k(0)+ M^{\mathcal{Y},k}(t) + N^{\mathcal{Y},k}(t); 0\le t\le T,$$
where $M^{\mathcal{Y},k}$ is an $\mathbb{F}^k$-square-integrable martingale and $N^{\mathcal{Y},k}$ is an $\mathbb{F}^k$-predictable absolutely continuous process.
\begin{definition}\label{stabledef}
Let $X=X(0) + M + N$ be an $\mathbb{F}$-adapted process such that $M\in \mathbf{H}^2(\mathbb{F})$ and $N\in \mathbf{B}^2(\mathbb{F})$ has continuous paths. An imbedded discrete structure $\mathcal{Y}=\big( (X^k)_{k\ge 1}; \mathscr{D}\big)$ for $X$ is said to be \textbf{stable} if $M^{\mathcal{Y},k}\rightarrow M$ weakly in $\mathbf{B}^2(\mathbb{F})$ as $k\rightarrow+\infty$.
\end{definition}

Let us now devote our attention to the study of the asymptotic properties of

$$\mathbb{D}^{\mathcal{Y},k}X$$
for an imbedded discrete structure $\mathcal{Y} = \big((X^k)_{k\ge 1}, \mathscr{D}\big)$ w.r.t $X$. We set

\begin{equation}\label{DX}
\mathcal{D}^{\mathcal{Y}}X:=\lim_{k\rightarrow+\infty}\mathbb{D}^{\mathcal{Y},k}X~\text{weakly in}~L^2_a(\mathbb{P}\times Leb)
\end{equation}
whenever the right-hand side of (\ref{DX}) exists for a given finite-energy embedded discrete structure $\mathcal{Y}$.

\begin{definition}\label{defweakder}
Let $X= X(0) + \int HdB + V$ be an $\mathbb{F}$-adapted process such that $H\in L^2_a(\mathbb{P}\times Leb)$ and $V\in \mathbf{B}^2(\mathbb{F})$ has continuous paths. We say that $X$ is \textbf{weakly differentiable} if there exists a finite energy imbedded discrete structure $\mathcal{Y} =\big((X^k)_{k\ge 1},\mathscr{D}\big)$ such that $\mathcal{D}^{\mathcal{Y}}X$ exists. The space of weakly differentiable processes will be denoted by $\mathcal{W}(\mathbb{F})$.
\end{definition}
If there exists a strictly increasing sequence of stopping times $\{T_N;N\ge 1\}$ such that the stopped process $X(\cdot \wedge T_N)\in \mathcal{W}(\mathbb{F})$ for every $N\ge 1$ and $T_N\uparrow+\infty~a.s$ as $N\rightarrow+\infty$, we then say that $X$ is locally weakly differentiable.

\begin{remark}\label{indeps}
Theorem 4.1 in \cite{LOS} shows that if $X$ is a weakly differentiable Wiener functional, then it has the form
$$X(t) = X(0) + \int_0^tH(s)dB(s) + V(t); 0\le t\le T,$$
where $H\in L^2_a(\mathbb{P}\times Leb)$ and $V\in \mathbf{B}^2(\mathbb{F})$ has continuous paths. Moreover, $\mathcal{D}^\mathcal{Y}X = H$ for every stable imbedded discrete structure $\mathcal{Y}$ associated with $X$.
\end{remark}

Let $X\in \mathcal{W}(\mathbb{F})$ with a decomposition

$$
X = X(0)+ \int H dB + V
$$
where $H\in L^2_a(\mathbb{P}\times Leb)$ and $V\in \mathbf{B}^2(\mathbb{F})$ is a continuous process. Then, Remark \ref{indeps} allows us to define

$$
\mathcal{D}X := \mathcal{D}^\mathcal{Y}X
$$
for every stable imbedded discrete structure $\mathcal{Y} =\big((X^k)_{k\ge 1},\mathscr{D}\big)$ w.r.t $X$ and, in this case, $\mathcal{D}X = H$. The weak differentiability notion requires existence of $\mathcal{D}^\mathcal{Y}X$ for a finite energy imbedded discrete structure $\mathcal{Y} = \big((X^k)_{k\ge 1},\mathscr{D}\big)$ and, from Remark \ref{indeps}, this concept of derivative does not depend on the choice of the stable structure $\mathcal{Y}$.

Let us now finish this section by recalling the following result.

\

\textbf{Proposition 4.2} of~\cite{LOS}: If $X\in \mathcal{W}(\mathbb{F})$ is a weakly differentiable process with a decomposition

$$X(t) = X(0) + \int_0^t\mathcal{D}X(s)dB(s) + V_X(t); 0\le t\le T,$$
then

\begin{equation}\label{allort}
V_X(\cdot)  = \lim_{k\rightarrow+\infty}\int_0^\cdot\mathbb{U}^{\mathcal{Y},k}X(s)ds
\end{equation}
weakly in $\mathbf{B}^2(\mathbb{F})$ as $k\rightarrow +\infty$ for every stable imbedded discrete structure $\mathcal{Y}$ associated with $X$.

\

The remainder of this paper is devoted to investigate the probabilistic structure of the limit (\ref{allort}). In \cite{LOS}, the authors investigate in detail many types of drifts encoded by (\ref{allort}). For instance, the existence of the limit $\lim_{k\rightarrow+\infty}\mathbb{U}^{\mathcal{Y},k}X$ (weakly in $L^1_a(\mathbb{P}\times Leb)$) implies $X$ is an It\^o process and, more generally, finite $1\le p < 2$-variation drifts of strong Dirichlet processes are determined by (\ref{allort}). The next sections aims to deepen the analysis of (\ref{allort}) towards a variational characterization of drifts exploring the concepts of local-times, horizontal-type and first-order vertical-type perturbations.

\section{Occupation times and weak differentiability}\label{LOCALTIMESEC}
In this section, we aim to present a universal characterization of the non-martingale component $V_X$ of a given weakly differentiable process $X$ by exploring the relation between weak differentiability and a suitable two-parameter \textit{occupation time} process associated with $\mathscr{D}$.
It is relatively well-understood the role of the horizontal and second-order vertical derivatives in It\^o-type decompositions of smooth functionals applied to state processes (even non-semimartingales \cite{keller,cont1}). The weak formulations of \cite{cont2,peng2} based on abstract completion on cylindrical functionals only deals with situations where the drift term is absolutely continuous. In the remainder of this article, we show the class of weak differentiable processes contains a large class of Wiener functionals with irregular drift components. The typical example we have in mind are processes whose drift admits $q$-finite variation for $2\le q < \infty$. See Theorem \ref{ulocaltime} and Section \ref{pqregularitysection}.

\subsection{Splitting the weak infinitesimal generator}
Let $C([0,t];\mathbb{R})$ ($D([0,t];\mathbb{R})$) be the spaces of real-valued continuous (c\`adl\`ag) functions defined over $[0,t]$. We say that $F:[0,T]\times D([0,T];\mathbb{R})\rightarrow\mathbb{R}$ is a \textit{non-anticipative functional} if

\begin{equation}\label{nonantiintr}
F_t(c) = F_t(c_t); 0\le t\le T
\end{equation}
for each $c\in D([0,T];\mathbb{R})$ where $c_t: = c(\cdot \wedge t); 0\le t\le T$. In order to make clear the information encoded by a path $c \in D([0,T];\mathbb{R})$ up to a given time $0\le t\le T$, we then view $c_t= \{c(s): 0 \leq s \leq t \}$ and the value of $c$ at time $0 \leq u \leq T$ is denoted by $c(u)$. This notation is naturally extended to adapted processes. In the sequel, for a given $s\in (0,T]$, let us define

$$D([0,s];\mathbb{R})\times \mathbb{R}\rightarrow D([0,s];\mathbb{R})$$
$$(\eta,x)\mapsto \textbf{t}(\eta_s, x)$$
where we set
\begin{equation}\label{terminalM}
\textbf{t}(\eta_s,x)(u):=\left\{
\begin{array}{rl}
\eta(u); & \hbox{if}~0\le u<s \\
x;& \hbox{if}~u=s.
\end{array}
\right.
\end{equation}
For any Wiener functional $X$, Doob-Dynkin's theorem yields the existence of a non-anticipative functional~\footnote{Doob-Dynkin's theorem yields the existence of a non-anticipative functional $\hat{F}:[0,T]\times C([0,T];\mathbb{R})\rightarrow\mathbb{R}$ such that $X(t)=F_t(B_t)$ and we assume there exists a non-anticipative functional $F:[0,T]\times D([0,T];\mathbb{R})\rightarrow\mathbb{R}$ which is consistent to $\hat{F}$ in the sense that $F = \hat{F}$ on $[0,T]\times C([0,T];\mathbb{R})$.} $F$ such that

$$X(t) =F_t(B_t) = F_t(\textbf{t}(B_t, B(t)))~a.s, 0\le t\le T,$$
where $B(t)$ is the value of the Brownian path at time $t$.

At first, we need to split the $\mathbb{F}^k$-weak infinitesimal generator (\ref{weakinfG}) into two components. For this purpose, we set $\Lambda:=\{(t,\omega_t); (t,\omega)\in [0,T]\times D([0,T];\mathbb{R})\}$ and $\hat{\Lambda}: = \{(t,\omega_t); (t,\omega)\in [0,T]\times C([0,T];\mathbb{R})\}$. Let $X$ be a Wiener functional and let $\mathcal{Y}= \big((X^k)_{k\ge 1},\mathscr{D}\big)$ be an imbedded discrete structure for $X$. By Doob-Dynkin's theorem, there exist functionals $F$ and $F^k$ defined on $\hat{\Lambda}$ and $\Lambda$, respectively, such that
$\hat{F}_t(B_t)=  X(t); 0\le t\le T$ and

\begin{equation}\label{funcrep}
F^k_t(A^k_t)= X^k(t);~0\le t\le T.
\end{equation}
In the remainder of the paper, when we write $F(B)$ for a given non-anticipative functional $F$ defined on $\Lambda$ it is implicitly assumed that we are fixing a functional $F$ which is consistent to $\hat{F}$ in the sense that
$$F_t(x_t) = \hat{F}_t(x_t)~\text{for every}~x\in C([0,T];\mathbb{R}).$$
Since we are dealing only with adapted processes, throughout this section we assume that all functionals are non-anticipative. Let us fix $X\in \mathbf{B}^2(\mathbb{F})$ and functionals $F^k$ and $F$ defined on $\Lambda$ and realizing, respectively, (\ref{funcrep}) and

\begin{equation}\label{functionalrepresentation}
X(t) = F_t(B_t); 0\le t\le T.
\end{equation}
The first task is to split $U^{k,\mathcal{Y}}X$ into two components which encode different modes of regularity. In the sequel, it is convenient to introduce the following perturbation scheme

$$A^{k, i\epsilon_k}_{t-} (u):= A^k(u),~0 \le u < t\quad \text{and}~A^{k, i\epsilon_k}_{t-} (t) := A^k (t-) + i\epsilon_k,$$
for $i=-1 ,0,1$ and to shorten notation, we denote $A^{k}_{t-} := A^{k,0}_{t-}$ for $0\le t\le T.$ Since $X^k$ is a pure jump process, then

\begin{equation}\label{natgood}
X^k(t) = F^k_t(A^k_t)= \sum_{\ell=0}^\infty F^k_{T^k_\ell}(A^k_{T^k_\ell})1\!\!1_{\{T^{k}_{\ell}\le t < T^{k}_{\ell + 1}\}},~0\le t\le T.
\end{equation}

A simple calculation based on the i.i.d family $\Delta A^k(T^k_n)$ of Bernoulli variables with parameter $1/2$ yields the following splitting.

\begin{lemma}\label{splittinglemma}
If $\mathcal{Y} = \big((X^k)_{k\ge1},\mathscr{D}\big)$ is an imbedded discrete structure (with a functional representation~(\ref{funcrep})) associated with $X$, then
$$
U^{\mathcal{Y},k}X = \mathcal{D}^{\mathcal{Y},k,h}X + \frac{1}{2}\mathcal{D}^{\mathcal{Y},k,2}X
$$
where

\begin{equation}\label{ddkh}
\mathcal{D}^{\mathcal{Y},k,h}X(t):=
\sum_{n=1}^\infty \frac{1}{\epsilon^2_k} [ F^k_{t} (A^{k}_{t-})  - F^k_{T^k_{n-1}} (A^k_{T^k_{n-1}}) ]
1\!\!1_{\{ T_{n-1}^k <t \le T_{n}^k\}}
\end{equation}
and
\begin{equation}\label{ddk2}
\mathcal{D}^{\mathcal{Y},k,2}X(t):=
\frac{1}{\epsilon^2_k}  [F^k_{t} (A^{k, \epsilon_k}_{t-}) + F^k_{t} (A^{k, -\epsilon_k}_{t-}) - 2F^k_{t} (A^{k}_{t-}) ];~0\le t\le T
\end{equation}
are $\mathbb{F}^k$-predictable processes.
\end{lemma}
\begin{proof}
The proof is a simple computation as follows: From Lemma 3.2 in \cite{LOS}, we know that

$$U^{\mathcal{Y},k} X(t) = \frac{1}{\epsilon^2_k}\sum_{n=1}^\infty
\mathbb{E}[F^k_{T^k_n} (A^k_{T^k_n}) - F^k_{T^k_{n-1}} (A^k_{T^k_{n-1}})| \mathcal{F}^k_{T^k_{n-1}}; T^k_n=t]
1\!\!1_{\{ T_{n-1}^k <t \le T_{n}^k\}}$$
for $t>0$. Then, for $t>0$ and $k\ge 1$, we have

\begin{eqnarray*}
 U^{\mathcal{Y},k} X(t) &=& \frac{1}{\epsilon^2_k}\sum_{n=1}^\infty
\mathbb{E}[F^k_{T^k_n} (A^k_{T^k_n}) - F^k_{T^k_{n-1}} (A^k_{T^k_{n-1}})| \mathcal{F}^k_{T^k_{n-1}}; T^k_n=t]
1\!\!1_{\{ T_{n-1}^k <t \le T_{n}^k\}} \\
&=& \frac{1}{\epsilon^2_k}\sum_{n=1}^\infty
\mathbb{E}[F^k_{t} (A^k_{t}) - F^k_{T^k_{n-1}} (A^k_{T^k_{n-1}})| \mathcal{F}^k_{T^k_{n-1}}; T^k_n=t]
1\!\!1_{\{ T_{n-1}^k <t \le T_{n}^k\}} \\
&=&  \frac{1}{\epsilon^2_k}\sum_{n=1}^\infty
\frac{1}{2} \Big[F^k_t (A^{k, \epsilon_k}_{t^{-}}) + F^k_t (A^{k, -\epsilon_k}_{t^{-}}) - 2F^k_{t} (A^{k}_{t^{-}}) \\
&+&2F^k_{t} (A^{k}_{t^{-}})  - 2F^k_{T^k_{n-1}} (A^k_{T^k_{n-1}}) \Big]
1\!\!1_{\{ T_{n-1}^k <t \le T_{n}^k\}} \\
&=& \frac{1}{2}
\frac{[F^k_t (A^{k, \epsilon_k}_{t^{-}}) + F^k_t (A^{k, -\epsilon_k}_{t^{-}}) - 2F^k_{t} (A^{k}_{t^{-}})]}{\epsilon_k^2}\\
& &\\
&+& \sum_{n=1}^\infty \frac{1}{\epsilon^2_k} [ F^k_{t} (A^{k}_{t-}  - F^k_{T^k_{n-1}} (A^k_{T^k_{n-1}}) ]
1\!\!1_{\{ T_{n-1}^k <t \le T_{n}^k\}}.
\end{eqnarray*}
Let us now check $\mathbb{F}^k$-predictability. For each $n,k\ge 1$, we define

$$
A^{k,\infty}_{T^k_{n-1}}(u):=\left\{
\begin{array}{rl}
A^k(u); & \hbox{if} \ 0\le u < T^k_{n-1} \\
A^k(T^k_{n-1});& \hbox{if} \ T^k_{n-1}\le u < \infty \\
\end{array}
\right.
$$
and we set $A^{k,\infty}_{T^k_{n-1},t}:=\text{restriction of}~A^{k,\infty}_{T^k_{n-1}}~\text{over the interval}~[0,t]$. We also define

$$
A^{k, i\epsilon_k}_{T^k_{n-1},t} (u):=\left\{
\begin{array}{rl}
A^{k,\infty}_{T^k_{n-1},t}(u); & \hbox{if} \ 0\le u < t \\
A^k (T^k_{n-1}) + i\epsilon_k;& \hbox{if} \ u=t \\
\end{array}
\right.
$$
where $i=-1,0,1$ and we denote $A^{k}_{T^k_{n-1},t}  = A^{k, 0}_{T^k_{n-1},t}; n,k\ge 1$. By the very definition,

$$A^{k,i\epsilon_k}_{t-} = A^{k,i\epsilon_k}_{T^k_{n-1},t} ~\text{on}~ \{T^k_{n-1} < t \le T^k_{n}\}$$
and $F^k_\cdot (A^{k,i\epsilon_k}_{T^k_{n-1},\cdot})$ is $\mathcal{F}^k_{T^k_{n-1}}\otimes \mathcal{B}(\mathbb{R}_+)$-measurable for every $k,n\ge 1$ and $i=-1,0,1$. Therefore, $F^k_t(A^{k,i\epsilon_k}_{t-})= F^k_t(A^{k,i\epsilon_k}_{T^k_{n-1},t})$ on $\{T^k_{n-1}< t \le T^k_n\}$ for every $k,n\ge 1$ and hence, we shall write

$$\mathcal{D}^{\mathcal{Y},k,h}X(t)=
\sum_{n=1}^\infty \frac{1}{\epsilon^2_k} [ F^k_t(A^{k}_{T^k_{n-1},t})  - F^k_{T^k_{n-1}} (A^k_{T^k_{n-1}}) ]
1\!\!1_{\{ T_{n-1}^k <t \le T_{n}^k\}}
$$
and

$$\mathcal{D}^{\mathcal{Y},k,2}X(t)=
\sum_{n=1}^\infty \frac{1}{\epsilon^2_k} [F^k_t (A^{k, \epsilon_k}_{T^k_{n-1},t}) + F^k_t (A^{k, -\epsilon_k}_{T^k_{n-1},t}) - 2F^k_t (A^{k}_{T^k_{n-1},t})]
1\!\!1_{\{ T^k_{n-1} <t \le T^k_{n}\}}.
$$
Now we may conclude (see e.g Th 5.5 in \cite{he}) that both $\mathcal{D}^{\mathcal{Y},k,h}X$ and $\mathcal{D}^{\mathcal{Y},k,2}X$ are $\mathbb{F}^k$-predictable processes.
\end{proof}

\begin{remark}
The operators $\mathcal{D}^{\mathcal{Y},k,h}$ and $\mathcal{D}^{\mathcal{Y},k,2}$ encode variation in the same spirit of the horizontal and second order vertical derivatives used in the pathwise functional calculus, but with one fundamental difference: In contrast to the pathwise calculus where shifts are deterministic, the increments in the operators~(\ref{ddkh}) and~(\ref{ddk2}) are driven by the stopping times $\{T^k_n; n\ge 1\}$ and the Bernoulli variables $\{\Delta A^k(T^k_n); n\ge 1\}$.
\end{remark}

\begin{example} Let $\mathcal{F}=\big((\textbf{F}^k)_{k\ge 1}, \mathscr{D}\big)$ be the functional imbedded discrete structure associated with $X=F(B)$ defined by

\begin{equation}\label{funcSTRUCTURE}
\mathbf{F}^k(t)=\sum_{n=0}^\infty F_{T^k_n}(A^k_{T^k_n})1\!\!1_{\{ T_{n}^k \le t \le T_{n+1}^k\}};0\le t\le T.
\end{equation}
If $F$ is a $\Lambda$-continuous non-anticipative functional in the sense of \cite{dupire} and $\sup_{k\ge 1}\|F(A^k)\|_{\mathbf{B}^2}< \infty$,  then $\lim_{k\rightarrow+\infty}\mathbf{F}^k = X$ weakly in $\mathbf{B}^2(\mathbb{F})$. For this type of imbedded structure, we have

$$
\mathcal{D}^{\mathcal{F},k,h}X(t)=
\sum_{n=1}^\infty \frac{1}{\epsilon^2_k} [ F_{t} (A^{k}_{t-})  - F_{T^k_{n-1}} (A^k_{T^k_{n-1}}) ]
1\!\!1_{\{ T_{n-1}^k <t \le T_{n}^k\}}$$
and
$$\mathcal{D}^{\mathcal{F},k,2}X(t)=
\frac{1}{\epsilon^2_k}  [F_{t} (A^{k, \epsilon_k}_{t-}) + F_{t} (A^{k, -\epsilon_k}_{t-}) - 2F_{t} (A^{k}_{t-}) ];~0\le t\le T.$$
\end{example}

For a given structure $\mathcal{Y}$ w.r.t $X$, we denote

\begin{equation}\label{ykh}
\mathbb{D}^{\mathcal{Y},k,h}X: = \mathcal{D}^{\mathcal{Y},k,h}X\frac{d \langle A^{k}, A^{k}\rangle}{dt}\quad\mathbb{D}^{\mathcal{Y},k,2}X: = \mathcal{D}^{\mathcal{Y},k,2}X\frac{d \langle A^{k}, A^{k}\rangle}{dt}.
\end{equation}
The following corollary is a simple consequence of Lemma \ref{splittinglemma} and (\ref{discretediffform}).

\begin{corollary}\label{reprewithsplit}
Let $\mathcal{Y} = \big((X^k)_{k\ge1},\mathscr{D}\big)$ be an imbedded discrete structure (with functional representation (\ref{funcrep})) associated with $X$. If

\begin{equation}\label{i1}
\mathbb{E}\int_0^T\big(|\mathcal{D}^{\mathcal{Y},k,2}X(s)| + |\mathcal{D}^{\mathcal{Y},k,h}X(s)|\big)d\langle A^k, A^k\rangle(s) <\infty;~k\ge 1,
\end{equation}
then

$$
X^k(t) = X^k(0) + \oint_0^t\mathbb{D}^{\mathcal{Y},k}X(s) dA^k(s) + \int_0^t \Big(\mathbb{D}^{\mathcal{Y},k,h}X(s)+
\frac{1}{2}\mathbb{D}^{\mathcal{Y},k,2}X(s)\Big)ds
$$
for $0\le t\le T$, $k\ge 1$.
\end{corollary}
Let us denote

\begin{equation}\label{SMOOTHDER}
\mathcal{D}^{h,\mathcal{Y}}X:=\lim_{k\rightarrow\infty}\mathbb{D}^{\mathcal{Y},k,h}X, \quad \mathcal{D}^2X:=\lim_{k\rightarrow\infty}\mathbb{D}^{\mathcal{Y},k,2}X
\end{equation}
weakly in $L^1_a(\mathbb{P}\times Leb)$. The existence of the limits (\ref{SMOOTHDER}) require strong regularity of $X\in \mathcal{W}(\mathbb{F})$ (see Remark 4.5 in \cite{LOS}) and it is reminiscent from smooth $C^{1,2}(\mathbb{R}_+\times\mathbb{R})$-transformations of semimartingales in the non-path-dependent case. In the path-dependent case, such regularity rarely exists and this is one of the motivations to the development of different notions of viscosity methods in path-dependent PDEs. See \cite{peng1,touzi1,touzi2,touzi3,cosso1} for more details.

\subsection{Universal characterization of non-martingale components for weakly differentiable processes}
In this section, we present a universal variational characterization of the drift term $V_X$ of a given $X\in \mathcal{W}(\mathbb{F})$ with a decomposition

$$X(t) = X(0) + \int_0^t\mathcal{D}X(s)dB(s) + V_X(t); 0\le t\le T.$$
We observe that a priori $V_X$ is just a continuous process. The goal here is to represent $V_X$ by means of limits of integral functionals based on a given stable imbedded discrete structure $\mathcal{Y}  = \big((X^k)_{k\ge 1}, \mathscr{D}\big)$ w.r.t $X$. Under rather weak regularity conditions, it turns out that $V_X$ will be described by suitable occupation times of the driving discrete-type skeleton combined with a first-order operator. In order to start the analysis, let us consider 

$$
\textbf{t}(A^{k}_s,(j+i)\epsilon_k)=\left\{
\begin{array}{rl}
A^k(u); & \hbox{if}~0\le u<s \\
(j+i)\epsilon_k;& \hbox{if}~u=s
\end{array}
\right.
$$
for $0< s\le T$, $j\in \mathbb{Z}$ and $i=-1,0,1$. With these objects at hand, if $X\in \mathbf{B}^2(\mathbb{F})$ admits an imbedded discrete structure $\mathcal{Y} = \big((X^k))_{k\ge 1}, \mathscr{D}\big)$ with a given functional representation~(\ref{funcrep}), then we define

\begin{equation}\label{fkdd}
\nabla^{\mathcal{Y},k,v}X(s,x):= \sum_{j\in \mathbb{Z}} \frac{F^k_s(\textbf{t}(A^{k}_s,j\epsilon_k)) - F^k_s(\textbf{t}(A^{k}_s,(j-1)\epsilon_k))}  {\epsilon_k} 1\!\!1_{S^k_j}(x);~(s,x)\in [0,T]\times \mathbb{R}
\end{equation}
where $S^k_j:= \big((j-1)\epsilon_k, j\epsilon_k\big]$ for $j\in \mathbb{Z}$ and $k\ge 1$. For instance, if we take the associated functional imbedded discrete structure $\mathcal{F} = \big((\textbf{F}^k)_{k\ge 1},\mathscr{D}\big)$ given by (\ref{funcSTRUCTURE}), then

$$
\nabla^{\mathcal{F},k,v}X(s,x)= \sum_{j\in \mathbb{Z}} \frac{F_s(\textbf{t}(A^{k}_s,j\epsilon_k)) - F_s(\textbf{t}(A^{k}_s,(j-1)\epsilon_k))}  {\epsilon_k} 1\!\!1_{S^k_j}(x); (s,x)\in [0,T]\times \mathbb{R}.
$$
In order to switch second-order derivative $\mathcal{D}^{\mathcal{Y},k,2}$ into~(\ref{fkdd}), the following natural notion of occupation times plays a key role

\begin{equation}\label{Lk}
\mathbb{L}^{k,x}(t) :=\sum_{j\in \mathbb{Z}}\ell^{k,j\epsilon_k}(t) 1\!\!1_{S^k_j}(x); (t,x)\in [0,T]\times \mathbb{R},
\end{equation}
where
\begin{equation}\label{ocid1}
\ell^{k,x}(t):= \frac{1}{\epsilon_k}\int_0^t 1\!\!1_{\{|A^ k(s-) - x| < \epsilon_k \}}d\langle A^k,A^k\rangle(s);~k\ge 1,~x\in \mathbb{R},~ 0\le t \le T.
\end{equation}
Instead of Lebesgue measure, the occupation time~(\ref{ocid1}) is computed by a different clock $\langle A^k,A^k\rangle$ which has absolutely continuous paths by Lemma 2.1 in \cite{LOS}. In order to express our preliminary functional It\^o formula, we need a natural notion of integration w.r.t the occupation times
$\ell^{k}$. If $H=\{H(t,x); (t,x)\in [0,T]\times\mathbb{R}\}$ is a simple random field of the form

$$H(t,x) = \sum_{j\in \mathbb{Z}}\alpha_j(t)1\!\!1_{S^k_j}(x)$$
for some measurable process $\alpha_j:\Omega\times [0,T]\rightarrow \mathbb{R}$ then integration w.r.t $\Big\{\mathbb{L}^{k,x}(t); x\in\mathbb{R}, t\in [0,T]\Big\}$ is naturally defined by

\begin{equation}\label{intwrtl}
\int_0^t\int_{\mathbb{R}} H(s,x)d_{(s,x)}\mathbb{L}^{k,x}(s):= \sum_{j\in \mathbb{Z}}\int_0^t\alpha_j(s)\Big[d_s\ell^{k,j\epsilon_k}(s) -
d_s\ell^{k,(j-1)\epsilon_k}(s)    \Big]
\end{equation}
for $~0\le t\le T$, whenever the right-hand side of~(\ref{intwrtl}) is finite a.s.

\begin{lemma}\label{parts}
Let $X$ be a Wiener functional and let $\mathcal{Y}=\big((X^k)_{k\ge1},\mathscr{D}\big)$ be an imbedded discrete structure w.r.t $X$ with a functional representation~(\ref{funcrep}) and satisfying~(\ref{i1}). Then

$$
\frac{1}{2}\int_0^t \mathcal{D}^ {\mathcal{Y},k,2}X(s)d\langle A^k, A^k\rangle(s) =
-\frac{1}{2} \int_0^t\int_{-\infty}^{+\infty}\nabla^{\mathcal{Y},k,v}X(s,x))d_{(s,x)}\mathbb{L}^{k,x}(s)
$$
for $0\le t\le T$. Hence, the following decomposition holds

\begin{eqnarray}\label{fulldec}
X^k(t) &=& X^k(0) + \oint_0^t\mathbb{D}^{\mathcal{Y},k}X(s)dA^ k(s) + \int_0^t \mathbb{D}^{\mathcal{Y},k,h}X(s)ds\\
\nonumber& &\\
\nonumber&-& \frac{1}{2}\int_0^t\int_{-\infty}^{+\infty}\nabla^ {\mathcal{Y},k,v}X(s,x)d_{(s,x)}\mathbb{L}^{k,x}(s);~0\le t\le T,
\end{eqnarray}
for each $k\ge 1$.
\end{lemma}
\begin{proof}
In order to simplify the computations, we assume without any loss of generality that $\epsilon_k = 2^{-k}$. Let us introduce the following objects
\begin{equation}\label{deltaF}
\Delta F_{s,j,k}:=F^k_s(\textbf{t}(A^k_s,(j+1)2^ {-k})) - F^k_s(\textbf{t}(A^k_s,j2^ {-k}))
\end{equation}
$$b^{k,j}(s):=\frac{1}{2^{-k}}1\!\!1_{\{|A^ k(s-) - j2^{-k}| < 2^{-k} \}}
$$
and we set $\Delta b^{k,j+1}(s):= b^{k,j+1}(s)-b^{k,j}(s)$ for $k\ge 1,~j\in \mathbb{Z},~ 0\le s \le T$. For a given $j \in \mathbb{Z}$ and $s\in [0,T]$, we observe $\textbf{t}(A^{k}_s,(j+i)2^{-k}) = A^{k,i2^ {-k}}_{s-}$ on $\{ A^ k(s-) = j2^ {-k}\}$ for each $i=-1,0,1$. More importantly,

$$
\mathcal{D}^{\mathcal{Y},k,2} X(s) = \sum_{j\in \mathbb{Z}} \Bigg( \frac{\Delta F_{s,j,k} - \Delta F_{s,j-1,k}}{2^ {-k}}\Bigg) \frac{1}{2^ {-k}}1\!\!1_{\{A^ k(s-) = j2^ {-k} \} }~a.s;~0\le s\le T.
$$
For a given positive integer $m\ge 1$, let $J^k_m:=\inf\{0\le s\le T; |A^k(s)| >2^m\}$. Then

$$
\mathcal{D}^{\mathcal{Y},k,2} X(s) = \sum_{j=-2^{k+m}}^{2^{k+m}} \Bigg( \frac{\Delta F_{s,j,k} - \Delta F_{s,j-1,k}}{2^ {-k}}\Bigg) \frac{1}{2^ {-k}}1\!\!1_{\{A^ k(s-) = j2^ {-k} \} }
$$
a.s on $\{0\le s< J^k_m\}$. On a given set $\{0\le s < J^k_m\}$, we perform a pathwise summation by
parts on the set $[-2^m, 2^m]\cap \mathbb{Z}2^{-k}$ as follows

\begin{eqnarray*}
\frac{1}{2}\sum_{j=-2^{k+m}-1}^{2^{k+m}-1}\Big(\frac{\Delta F_{s,j+1,k}-\Delta F_{s,j,k}}{2^{-k}}\Big)
b^{k,j+1}(s)&=&\frac{1}{2}\Bigg[\frac{\Delta F_{s,2^{m+k},k}}{2^{-k}}b^{k,2^{m+k}}(s)\\
&-&\frac{\Delta F_{s,-2^{m+k}-1,k}}{2^{-k}}b^{k,-2^{m+k}-1}(s)\Bigg]\\
&-& \frac{1}{2} \sum_{j=-2^{k+m}-1}^{2^{k+m}-1} \frac{\Delta F_{s,j,k}}{2^{-k}}\Delta b^{k,j+1}(s)
\end{eqnarray*}
a.s where $b^{k,2^{k+m}}(s,\omega) =b^{k,2^{k+m-1}}(s,\omega)=0 $ if $0\le s < J^k_m(\omega)$. Then

\begin{equation}\label{intp}
\frac{1}{2}\sum_{j=-2^{k+m}-1}^{2^{k+m}-1}\Big(\frac{\Delta F_{s,j+1,k}-\Delta F_{s,j,k}}{2^{-k}}\Big)
b^{k,j+1}(s)=- \frac{1}{2} \sum_{j=-2^{k+m}-1}^{2^{k+m}-1} \frac{\Delta F_{s,j,k}}{2^{-k}}\Delta b^{k,j+1}(s)
\end{equation}
a.s for $0\le s < J^k_m$. For a given $0\le t \le T $,~(\ref{intp}) yields
\begin{equation}\label{mmm}
\frac{1}{2}\int_{[0, J^k_m\wedge t)}\mathcal{D}^ {\mathcal{Y},k,2}X(s) d\langle A^k,A^k\rangle(s)
=-\frac{1}{2}\int_{[0,J^k_m\wedge t)}\int_{-\infty}^{+\infty}\nabla^{\mathcal{Y},k,v}X(s,x)d_{(s,x)}\mathbb{L}^{k,x}(s)
\end{equation}
a.s for each $m\ge 1$. By taking $m\rightarrow \infty$ in~(\ref{mmm}), we conclude the proof.
\end{proof}
The following corollary reveals a broad connection between occupation times and weak differentiability of Wiener functionals.

\begin{proposition}\label{ulocaltime}
Let $X\in \mathcal{W}(\mathbb{F})$ be a weakly differentiable Wiener functional with a decomposition

\begin{equation}\label{decomposicao}
X(t) = X(0) + \int_0^t \mathcal{D}X(s)dB(s) + V_X(t)~0;\le t\le T.
\end{equation}
Then,

\begin{equation}\label{v}
V_X(\cdot)=\lim_{k\rightarrow \infty}\Bigg[\int_0^\cdot \mathbb{D}^{\mathcal{Y},k,h}X(s)ds - \frac{1}{2}\int_0^\cdot\int_{-\infty}^{+\infty}\nabla^ {\mathcal{Y},k,v}X(s,x)d_{(s,x)}\mathbb{L}^{k,x}(s)\Bigg]
\end{equation}
weakly in $\mathbf{B}^2(\mathbb{F})$ for every stable imbedded discrete structure $\mathcal{Y} = \big((X^k)_{k\ge1},\mathscr{D}\big)$ satisfying (\ref{i1}). In particular, for any $\mathbb{F}$-Dirichlet process $X$ in the sense of Bertoin \cite{bertoin} (e.g square-integrable $\mathbb{F}$-semimartingales) with canonical decomposition (\ref{decomposicao}), the zero energy component $V_X$ equals to (\ref{v}).
\end{proposition}
\begin{proof}
This is an immediate consequence of representation~(\ref{fulldec}) and the fact that both

$$\oint_0^{\cdot}\mathbb{D}^{\mathcal{Y},k}X(s)dA^ k(s) \rightarrow \int_0^{\cdot} \mathcal{D}X(s)dB(s)~\text{weakly in}~\mathbf{B}^2(\mathbb{F})$$
and
$$\mathbb{D}^{\mathcal{Y},k}X\rightarrow \mathcal{D}X~\text{weakly in}~L^2_a(\mathbb{P}\times Leb)$$
as $k\rightarrow+\infty$, for every stable imbedded discrete structure $\big((X^k)_{k\ge1},\mathscr{D}\big)$ w.r.t $X$ (see Remark \ref{indeps} and Prop. 4.2 in \cite{LOS}). The last assertion follows from Theorem 4.2 in \cite{LOS}.
\end{proof}

\section{Weak differentiability and finite $(p,q)$-variation regularity}\label{pqregularitysection}
In this section, we provide a class of examples of weakly differentiable Wiener functionals $X$ which exhibit rough dependence w.r.t Brownian paths. More precisely, let us assume the existence of a non-anticipative functional $F$ realizing $X(t) = F_t(\textbf{t}(B_t,B(t))); 0\le t\le T$, where the following regularity conditions are in force:

\

\noindent \textbf{Assumption A1:} For each $c\in D([0,T];\mathbb{R})$ and $t\in (0,T]$, $x\mapsto F_t(\textbf{t}(c_t,x))$ is $C^1(\mathbb{R})$,

$$(t,c,x)\mapsto F_t(\textbf{t}(c_t,x))$$
is jointly continuous and it has linear growth, i.e., for every compact set $J_1\times J_2\subset \mathbb{R}^2$, there exists a constant $M_1$ such that

$$|F_t(\textbf{t}(c_t,a))|\le M_1\big(1+ \sup_{0\le s\le t}|c(s\wedge t)|\big)$$
for every $c\in D([0,T];J_1), t\in [0,T]$ and $a\in J_2$.

\

\noindent \textbf{Assumption A2:} For every compact set $J\subset\mathbb{R}$, there exists a positive constant $M_2$ such that

$$|\nabla_x F_t(\textbf{t}(c_t,a)) - \nabla_x F_t(\textbf{t}(c_t,b))|\le M_2 |a-b|^{\gamma_1}$$
for every $t\in [0,T]$, $(a,b)\in \mathbb{R}^2$ and for every $c\in D([0,T];J)$, where $0< \gamma_1 \le 1$.

\

\noindent \textbf{Assumption A3:} For every compact set $J\subset\mathbb{R}$, there exists a positive constant $M_3$ such that

$$|\nabla_x F_t(\textbf{t}(c_t,a)) - \nabla_x F_t(\textbf{t}(d_t,a))|\le M_3 \sup_{0\le s\le T}|c(s)-d(s)|^{\gamma_2}$$
for every $t\in [0,T]$ and $a\in J$, where $0< \gamma_2 \le 1$.

\

\noindent \textbf{Assumption A4:} The map $(t,c,x)\mapsto \nabla_x F_t(\textbf{t}(c_t,x))$ is jointly continuous and there exists a constant $M_4$ such that
$$|\nabla_x F_t(\textbf{t}(c_t,a))|\le M_4\big(1+ \sup_{0\le s\le t}|c(s\wedge t)|\big)$$
for every $c\in D([0,T];\mathbb{R}), t\in [0,T]$ and $a\in \mathbb{R}$.

\

\noindent \textbf{Assumption A5:} For every compact set $J\subset\mathbb{R}$, there exists a positive constant $M_5$ such that

$$| F_{t+h}(c_{t,h}) - F_t(c_t)|\le M_5  h^{\gamma_3}$$
for every $t\in [0,T], h>0$ and $c\in D([0,T];J)$, where $0< \gamma_3 \le 1$.

\

In Assumption A5, $c_{t,h}\in D([0,t+h],\mathbb{R})$ is the horizontal extension of $c_t$ described as follows:
$$c_{t,h}(u):= c(u); 0\le u\le t $$
and $c_{t,h}(u):=c(t); t\le u \le t+h.$

\begin{remark}
Assumptions A1-A2-A3-A4-A5 do not imply horizontal and twice-vertical differentiability in the sense of \cite{dupire,cont}. See Example \ref{exampleYOUNG} for a concrete example.
\end{remark}

\begin{theorem}\label{FIRSTTH}
If $F$ satisfies Assumptions A1, A2, A3, A4 and A5, then $X=F(B)$ is a locally weakly differentiable Wiener functional where

$$\mathcal{D}X(t) = \nabla_x F(\textbf{t}(B_t,B(t)); 0\le t\le T$$
and the drift $V_X$ is given by (\ref{v}) for every stable imbedded discrete structure $\mathcal{Y} = \big((X^k)_{k\ge1},\mathscr{D}\big)$ satisfying (\ref{i1}).
\end{theorem}
\begin{proof}
Throughout this proof, $C$ is a constant which may defer from line to line. Without any loss of generality, we set $\epsilon_k=2^{-k}$ for integers $k\ge 1$.By considering $S_m = \inf\{t >0; |B(t)|\ge 2^m\}$, both $A^k$ and $B$ are bounded on the stochastic interval $[[0,S_m[[$ for $m,k\ge 1$. Moreover, since we shall write $2^m = j2^{-k}$ for $j=2^{m+k}$ ($m$ fixed), we observe that $S_m$ is an $\mathbb{F}^k$-stopping time for each $k\ge 1$. Therefore, by localization, for a given $m\ge 1$, we may assume that both $A^k$ and $B$ lie on the compact set $I_m = [-2^m,2^m]$ for every $k \ge 1$. The fact that $\mathcal{F}$ is an imbedded discrete structure for $F(B)$ is an immediate consequence of Assumption A1. We now claim that

\begin{equation}\label{Df}
\lim_{k\rightarrow+\infty}\mathbb{D}^{\mathcal{F},k}X(\cdot) = \nabla_x F_t(\textbf{t}(B_\cdot,B(\cdot))~\text{strongly in}~L^2_a(\mathbb{P}\times Leb).
\end{equation}
Let us denote $\eta^k_n:=A^k(T^k_{n-1})+ \Delta A^k(T^k_n)$. We notice that we shall write
$A^k_{T^k_n}= \textbf{t}(A^k_{T^k_n-},\eta^k_n )$ and $A^k_{T^k_n-} = \textbf{t}(A^k_{T^k_n-},A^k(T^k_{n-1}))$ so that

\begin{eqnarray*}
\mathbb{D}^{\mathcal{F},k}X(t)&=&\sum_{n=1}^{\infty}  \Bigg( F_{T^k_n}(\textbf{t}(A^k_{T^k_n-}, \eta^k_n)) - F_{T^k_{n}}(A^k_{T^k_{n-}})\Bigg)\frac{1}{\Delta A^k(T^k_n)}
1\!\!1_{\{T^k_{n}\le t < T^k_{n+1}\}}\\
& &\\
&+& \sum_{n=1}^{\infty}  \Big( F_{T^k_n}(A^{k}_{T^k_{n}-}) - F_{T^k_{n-1}}(A^k_{T^k_{n-1}})\Big)\frac{1}{\Delta A^k(T^k_n)}
1\!\!1_{\{T^k_{n}\le t < T^k_{n+1}\}}\\
& &\\
&=:&I^{k,1}(t) + I^{k,2}(t);~0\le t\le T.
\end{eqnarray*}
Because $x\mapsto F_t(\textbf{t}(c_t,x))$ in $C^1(\mathbb{R})$ for each path $c\in D([0,T];\mathbb{R})$, we shall apply the usual mean value theorem to get the existence of a family of random variables $ 0 <\gamma_{k,n}< 1$ a.s such that

\begin{eqnarray*}
I^{k,1}(t)&=& \nabla_xF_{T^k_n}\big(\textbf{t}(A^k_{T^k_n-}, A^k(T^k_{n-1}) + \gamma_{k,n}\Delta A^k(T^k_n)\big)\\
&\text{on}&~\{T^k_n\le t< T^k_{n+1},\Delta A^k(T^k_n)>0\}\\
\end{eqnarray*}

\begin{eqnarray*}
I^{k,1}(t)&=& \nabla_xF_{T^k_n}\big(\textbf{t}(A^k_{T^k_n-}, A^k(T^k_{n-1}) + \Delta A^k(T^k_n) - \gamma_{k,n}\Delta A^k(T^k_n)\big)\\
&\text{on}&~\{T^k_{n}\le t < T^k_{n+1},\Delta A^k(T^k_n)<0\}.
\end{eqnarray*}
We claim that $I^{k,1}(\cdot)\rightarrow \nabla_xF_\cdot(\textbf{t}(B_\cdot,B(\cdot)))$ weakly in $L^2_a(\mathbb{P}\times Leb)$ as $k\rightarrow \infty$. To shorten notation, let us denote

$$
a_{k,n}:=\left\{
\begin{array}{rl}
A^k(T^k_{n-1}) + \gamma_{k,n}\Delta A^k(T^k_n); & \hbox{if} \ \Delta A^{k,j}(T^{k,j}_n) > 0 \\
A^k(T^k_{n-1}) + \Delta A^k(T^k_n) - \gamma_{k,n}\Delta A^k(T^k_n);& \hbox{if} \ \Delta A^{k,j}(T^{k,j}_n)< 0. \\
\end{array}
\right.
$$
for $k,n\ge 1$.


\begin{eqnarray*}
\nonumber\Big|I^{k,1}(t)  -\nabla_xF_t(\textbf{t}(B_t,B(t)))\Big| &\le& |I^{k,1}(t) - \nabla_xF_{T^k_n}(\textbf{t}(B_{T^k_n},B(T^k_n)))|\\
\nonumber& &\\
\nonumber&+& |\nabla_xF_{T^k_n}(\textbf{t}(B_{T^k_n},B(T^k_n))) - \nabla_xF_{t}(\textbf{t}(B_{t},B(t)))| \\
&=:& J^{k,1}(t) + J^{k,2}(t)
\end{eqnarray*}
on $\{T^k_n \le t <T^k_{n+1}\}$. Triangle inequality, (A2-A3), and the fact that $a_{k,n}$ and $B$ lies on the compact set $I_m$ a.s over $[0,T]$ imply the existence of constants $M_2$ and $M_3$ which do not depend on $t\in [0,T]$ and $k\ge 1$ such that

\begin{eqnarray*}
J^{k,1}(t) &\le& |\nabla_x F_{T^k_n}\big(\textbf{t}(A^k_{T^k_n-},a_{k,n})\big) -\nabla_x F_{T^k_n}\big(\textbf{t}(B_{T^k_n},a_{k,n})\big) | \\
&+&|\nabla_x F_{T^k_n}\big(\textbf{t}(B_{T^k_n},a_{k,n})\big) - \nabla_x F_{T^k_n}\big(\textbf{t}(B_{T^k_n},B(T^k_n))\big)\\
& \le& M_3 \sup_{0\le  u\le t}|A_{T^k_n-}(u) - B(u)|^{\gamma_2} + M_2|a_{k,n} - B(T^k_n)|^{\gamma_1}
\end{eqnarray*}
on $\{T^k_n \le t <T^k_{n+1}\}$. We have $\sup_{0\le  u\le t}|A_{T^k_n-}(u) - B(u)|\le 22^{-k}$ on $\{T^k_n \le t< T^k_{n+1}\}$, $|\Delta A^k(T^k_n)|\le 2^{-k}$ a.s and $0 < \gamma_{k,n} < 1$ a.s for every $k,n\ge1$. We then conclude that $\lim_{k\rightarrow \infty}J^{k,1} = 0$ strongly in $L^2_a(\mathbb{P}\times Leb)$. Assumption A4 yields $\lim_{k\rightarrow+\infty}J^{k,2}(t)=0$ a.s for each $t\in [0,T]$ and the linear growth in A4 yields $\lim_{k\rightarrow+\infty}J^{k,2} = 0$ strongly in $L^2_a(\mathbb{P}\times Leb)$. By Assumption A5, we have

$$
|I^{k,2}(t)|^2\le M_5 2^{2k}|T^k_{N^(t)} - T^k_{N^k(t)-1}|^{2\gamma_3}~a.s
$$
for every $t\in [0,T]$ and $k\ge 1$, where we recall $N^k(t)$ is given by (\ref{numberhitting}). We now notice that Lemma 5.2 in \cite{LEAO_OHASHI2017.1} yields

$$\mathbb{E} |T^k_{N^k(t)} - T^k_{N^k(t)-1}|^q = C2^{-2kq}$$
for each $k\ge 1$ for a constant $C$ which does not depend on $k$ and $t\in [0,T]$. Hence,

$$\mathbb{E}|I^{k,2}(t)|^2\le C 2^{2k}2^{-2k 2\gamma_3}$$
for every $t\in [0,T]$ and $k\ge 1$. Therefore, $\lim_{k\rightarrow+\infty}\mathbb{E}\int_0^T |I^{k,2}(t)|^2dt = 0$. This shows (\ref{Df}).

Now, let us check that $\mathcal{F}$ has finite energy. We observe

\begin{eqnarray}
\nonumber\mathbb{E}\int_0^T|\mathbb{D}^{\mathcal{F},k}X(s) |^2ds &=& \mathbb{E}[M^{\mathcal{F},k},M^{\mathcal{F},k}](T)\\
\nonumber &-&  \mathbb{E}\sum_{n=1}^\infty
|\mathbb{D}^{\mathcal{F},k}X (T^{k}_n)|^2 ( T^{k}_{n+1} - T)1\!\!1_{\{T^{k}_n\le T < T^{k}_{n+1}\} }\\
\label{enerAlast}& &\\
\nonumber&=&\mathbb{E}[M^{\mathcal{F},k},M^{\mathcal{F},k}](T)\\
\nonumber &-& 2^{2k}\mathbb{E}\Big(|\Delta \mathbf{F}^k(T^k_{N^k(T)})|^2|T^k_{N^k(T)+1} - T|\Big)
\end{eqnarray}
for every $k\ge 1$. Since the left-hand side of (\ref{enerAlast}) is bounded due to previous step, we only need to show that the second part of the right-hand side of (\ref{enerAlast}) is bounded. The linear growth assumption in (A1) yields

\begin{eqnarray*}
|\Delta \mathbf{F}^k(T^k_{N^k(T)})|^2 |T^k_{N^k(T)+1} - T|&\le& \big(M_3(1+\sup_{0\le s\le T}|A^k(s)|)\big)^2|\\
&\times&|T^k_{N^k(T)+1} - T|~a.s
\end{eqnarray*}
for every $k\ge 1$ so that there exists a constant $C$ such that

$$\mathbb{E}\Big(|\Delta \mathbf{F}^k(T^k_{N^k(T)})|^2|T^k_{N^k(T)+1} - T|\Big) \le C \big(\mathbb{E} |T^k_{N^k(T)+1} - T^k_{N^k(T)}|^2\big)^{\frac{1}{2}}$$
for every $k\ge 1$. Again, Lemma 5.2 in \cite{LEAO_OHASHI2017.1}, we know that

$$\mathbb{E} |T^k_{N^k(T)+1} - T^k_{N^k(T)}|^2 = C2^{-4k}$$
for a constant $C$ which does not depend on $k$. Therefore,

\begin{equation}\label{enerAlastt}
\sup_{k\ge 1}2^{2k}\Big(\mathbb{E}|\Delta \mathbf{F}^k(T^k_{N^k(T)})|^2|T^k_{N^k(T)+1} - T|\Big)<\infty.
\end{equation}
The estimates (\ref{enerAlast}) and (\ref{enerAlastt}) allow us to conclude the finite energy of $\mathcal{F}$. This concludes that $X\in\mathcal{W}(\mathbb{F})$ (up to a localization) and the characterization of the drift is due to Proposition \ref{ulocaltime}.
\end{proof}

\begin{example}\label{exampleYOUNG}
\end{example}

Let
$$
F_t(c_t):=\int_{-\infty}^{c(t)}\int_0^t \varphi(c(s),y)dsdy
$$
for $c\in D([0,T];\mathbb{R})$, where $\varphi:\mathbb{R}^2\rightarrow \mathbb{R}$ is a two-parameter H\"{o}lder continuous function such that: (i) $\int_{\mathbb{R}}\int_0^T|\varphi(c(s),y)|dsdy< \infty$ for every $c\in D([0,T];\mathbb{R})$, (ii) for every compact set $K$, there exist constants $M_1$ and $M_2$

\begin{equation}\label{Holderpr}
|\varphi(a,x) - \varphi(a,y)|\le M_1|x-y|^{\gamma_1}~\text{and}~|\varphi(c,z) - \varphi(d,z)|\le M_2|c-d|^{\gamma_2},
\end{equation}
for every $(x,y,c,d)\in \mathbb{R}^4$ and for every $(a,z)\in K$, where $\gamma_1\in (0,1], \gamma_2\in (0,1]$ ($\varphi$ has $(\gamma_1,\gamma_2)$-bivariation in the sense of \cite{young1}) and (iii) there exists a constant $M_3$ such that
$$
|\varphi(a,x)|\le M_3(1 + |a|)~\forall (a,x)\in \mathbb{R}^2.
$$
Moreover, (iv) for every compact set $V_1\subset \mathbb{R}$, there exists a compact set $V_2$ such that $\{x; \varphi(a,x)\neq 0\}\subset V_2$ for every $a\in V_1$.

Assumption (ii) implies: For every $L>0$, there exists a constant $C>0$ such that
$$\big|\Delta_j \Delta_i\nabla_x F_{t_i}(\textbf{t}(c_{t_i},x_j))\big|\le C|t_i-t_{i-1}||x_j - x_{j-1}|^{\gamma_1}$$
for every partition $\{t_i\}_{i=0}^N\times \{x_j\}_{j=0}^{N'}$ of $[0,T]\times [-L,L]$ and for every $c\in D([0,T];J)$ restricted to a compact set $J\subset \mathbb{R}$. Therefore,

$$(t,x)\mapsto\nabla_x F(\textbf{t}(c_t,x)) = \int_0^t\varphi(c(s),x)ds$$
has joint finite $\frac{1}{\gamma_1}$-variation in the sense of Friz and Victoir \cite{friz}. Under these conditions, $X\in \mathcal{W}(\mathbb{F})$ (up to localization) and


$$\mathcal{D}X(\cdot) =\int_0^\cdot\varphi(B(r),B(\cdot))dr.$$

In particular, if $\varphi(a,x):=a g(x)$ where $g$ is a $\gamma_1$-H\"{o}lder continuous $(0< \gamma_1 <1)$ and nowhere differentiable function with compact support, then $F$ does not admit second-order vertical derivatives in the sense of \cite{dupire,cont}, but $F(B)$ is a locally weakly differentiable process. For a more explicit characterization of the drift $V_X$ in terms of a two-parameter Young integral, one has to impose stronger regularity $\frac{1}{2} < \gamma_1\le 1$. See Example \ref{lastexample} for further details.

\subsection{Drifts with finite $q$-variation regularity for $q\ge 2$}\label{roughcaseSUB}
Let us now present a class of examples which illustrates that the weak differentiability concept is not restricted to the class of Dirichlet processes (See Proposition \ref{ulocaltime} above and Theorem 4.2 in \cite{LOS}) and indeed it covers a larger class of Wiener functionals whose drift components exhibit rather weak path-regularity. At this point, it is convenient to introduce the $p$-variation topology. let $\mathcal{W}_p(a,b)$ be the space of real-valued functions $f:[a,b]\rightarrow\mathbb{R}$ such that

$$\|f\|^p_{[a,b];p}:= \sup_{\Pi}\sum_{x_i\in \Pi}|f(x_i)-f(x_{i-1})|^p <\infty$$
where $p\ge 1$ and the $\sup$ is taken over all partitions $\Pi$ of a compact set $[a,b]\subset \mathbb{R}$. Let

$$\mathbf{F}^k(t) = \sum_{n=0}^\infty F_{T^k_n}(A^k_{T^k_n})\mathds{1}_{\{T^k_n\le t< T^k_{n+1}\}}$$
be the functional imbedded discrete structure associated with a given $X\in \mathbf{B}^2(\mathbb{F})$. Let us consider the following list of assumptions:

\

\noindent \textbf{Assumption B1:} For each $c\in D([0,T];\mathbb{R})$ and $t\in [0,T]$, $x\mapsto F_t(\textbf{t}(c_t,x))$ is $C^1(\mathbb{R})$,
and $\mathbf{F}^k\rightarrow F(B)$ weakly in $\mathbf{B}^2$ and it has finite-energy $\sup_{k\ge 1}\mathbb{E}[\mathbf{F}^k,\mathbf{F}^k](T) < +\infty$.

\

\noindent \textbf{Assumption B2:} There exists a positive constant $C_2$ such that

$$|\nabla_x F_t(\textbf{t}(c_t,a)) - \nabla_x F_t(\textbf{t}(c_t,b))|\le C_2 |a-b|^{\gamma_1}$$
for every $t\in [0,T]$, $(a,b)\in \mathbb{R}^2$ and for every $c\in D([0,T];\mathbb{R})$, where $0< \gamma_1 \le 1$.

\

\noindent \textbf{Assumption B3:} There exists a positive constant $C_3$ such that

$$|\nabla_x F_t(\textbf{t}(c_t,a)) - \nabla_x F_t(\textbf{t}(d_t,a))|\le C_3 \sup_{0\le s\le T}|c(s)-d(s)|^{\gamma_2}$$
for every $t\in [0,T]$ and $a\in \mathbb{R}$, where $0< \gamma_2 \le 1$.

\

\noindent \textbf{Assumption B4:} The map $(t,c,x)\mapsto \nabla_x F_t(\textbf{t}(c_t,x))$ is jointly continuous and

$$\sum_{n=0}^\infty |\nabla_xF_{T^k_n}(\textbf{t}(A^k_{T^k_n},B(T^k_n)))|^2\mathds{1}_{\{T^k_n \le t < T^k_{n+1}\}}; k\ge 1$$
is uniformly integrable for every $t\in [0,T]$.

\

\noindent \textbf{Assumption B5:} There exists a positive constant $C_5$ such that

$$| F_{t+h}(c_{t,h}) - F_t(c_t)|\le C_5  h^{\gamma_3}$$
for every $t\in [0,T], h>0$ and $c\in D([0,T];\mathbb{R})$, where $0< \gamma_3 \le 1$.

\

\noindent \textbf{Assumption B6:}

$$\sup_{k\ge 1}\epsilon^{-2}_k\mathbb{E}\Big(|\Delta \mathbf{F}^k(T^k_{N^k(T)})|^2|T^k_{N^k(T)+1} - T|\Big)< \infty$$
where we recall $N^k(T)$ is given by (\ref{numberhitting}). 
\begin{proposition}\label{roughcase}
If $F$ satisfies Assumptions B1, B2, B3, B4, B5 and B6, then $X=F(B)$ is a weakly differentiable Wiener functional where

$$\mathcal{D}X(t) = \nabla_x F(\textbf{t}(B_t,B(t)); 0\le t\le T.$$
\end{proposition}
\begin{proof}
The proof is similar to the one given in Theorem \ref{FIRSTTH} so we omit the details.
\end{proof}

In the remainder of this section, let us devote our attention to the following class of Wiener functionals

\begin{equation}\label{roughX}
X(t) = \int_0^t Z_{s}(B_s)dB(s) + \int_0^t Y_{s}(B_s)dg_s(B_s);0\le t\le T,
\end{equation}
where $Y(B)\in \mathcal{W}_p(0,T)~a.s$ and $g(B)\in \mathcal{W}_q(0,T)~a.s$ such that $1/p + 1/q>1$. The case $1\le q < 2$ was studied in \cite{LOS} so that we will restrict the analysis to the less regular case $2 \le q  < \infty$. For simplicity of exposition, we suppose $g = g(B)$ is a deterministic $q$-finite variation continuous function. Under suitable regularity condition on $Z$ and $Y$, we will show that $X$ satisfies (B1-B2-B3-B4-B5-B6). Let us study

$$F_t(\eta_t) = M_t(\eta_t) + J_t(\eta_t)$$
where $M$ and $J$ are non-anticipative functionals (to be defined in (\ref{pathwiseDRIFT}) and (\ref{pathwiseMARTINGALE})) describing the martingale and the drift, respectively in (\ref{roughX}). We assume that $Y(c)\in \mathcal{W}_p(0,T);\frac{1}{p} + \frac{1}{q}> 1$ for every $c\in D([0,T];\mathbb{R})$ and define $J(c)$ as a Young integral-type functional

\begin{equation}\label{pathwiseDRIFT}
J_t(c_t):=\int_0^t Y_s(c_s)dg(s); 0\le t\le T.
\end{equation}
Of course, $\nabla_x J_t(\mathbf{t}(c,x)) = 0$ so that $c\mapsto J(c)$ satisfies B2-B3-B4. By the modulus of continuity of the Young integral, $J(c)\in \mathcal{W}_q(0,T)$ for every $c\in D([0,T];\mathbb{R})$ so that B5 is fulfilled. The candidate for an imbedded structure w.r.t $\int Y(B)dg$ is

$$\mathbf{J}^k(t): = \sum_{n=0}^\infty \int_0^{T^k_n}Y_s(A^k_s)dg(s)\mathds{1}_{\{T^k_n\le t < T^k_{n+1}\}}; 0\le t\le T.$$
Let us assume

\begin{equation}\label{B7}
\lim_{k\rightarrow+\infty}\mathbb{E}\Big\|Y(A^k)\ - Y(B)\Big\|^2_{p;[0,T]}=0
\end{equation}
and there exists $\alpha >1$ such that

\begin{equation}\label{B8}
\sup_{k\ge 1}\mathbb{E}\Bigg|\int^{T^k_{N^k(T)}}_{T^k_{N^k(T)-1}}Y_s(A^k_s)dg(s)\Bigg|^{2\alpha} < \infty.
\end{equation}
By the Young-Lo\'eve estimate (see e.g \cite{young}), there exists a constant $C$ such that

\begin{equation}\label{youngest}
\sup_{0\le t\le t}\Big|\int_0^tY_s(B_s) - Y_s(A^k_s)dg(s)\Big|^2\le C \Big\| Y(A^k)\ - Y(B)\Big\|^2_{p;[0,T]} \big\| g\big\|^2_{q;[0,T]}~a.s.
\end{equation}
Then (\ref{B7}) implies $\lim_{k\rightarrow+\infty}\mathbf{J}^k = J(B)$ weakly in $\mathbf{B}^2(\mathbb{F})$. The $q$-H\"{o}lder modulus of continuity of the paths $t\mapsto J_t(c_t)$ allows us to make the same argument that we did for the term $I^{k,2}$ (in the proof of Theorem \ref{FIRSTTH}) to conclude that

$$\lim_{k\rightarrow+\infty}\mathbb{D}^{\mathbf{J},k}J(B) = 0$$
strongly in $L^2_a(\mathbb{P}\times Leb)$ where $\mathbf{J} = \{\mathbf{J}^k; k\ge 1\}$. This shows $\sup_{k\ge 1}\mathbb{E}[\mathbf{J}^k,\mathbf{J}^k](T) < \infty$ and we conclude $\mathbf{J} = \{\mathbf{J}^k; k\ge 1\}$ is an imbedded discrete structure for $\int Y_s(B_s)dg$ and B1 is fulfilled. Moreover, (\ref{B8}) jointly with Lemma 5.2 in \cite{LEAO_OHASHI2017.1} allow us to state that B6 holds true. Then, $J(B)$ satisfies the assumptions of Proposition \ref{roughcase}.

Let us now treat the martingale component. Let $q_n$ be a sequence such that $\sum_{n\ge 1}q^2_n < \infty$. Let $\rho, \eta\in D([0,T];\mathbb{R}) $ and we fix $ 0 < t \le T$. For $n\ge 1$, we set $a^n_0:=0$ and

$$a^n_{i+1}: = \inf\{t \ge a^n_i; |\rho(t) - \rho(a^n_i)|\ge q_n\}; i\ge 0.$$
We then define

\begin{equation}\label{pathwiseint}
\mathcal{I}_{t,n}\big(\rho_t, \eta_t\big):=\rho(0)\eta(0) + \sum_{i=0}^\infty\rho(a^n_i)\Big( \eta(a^n_{i+1}\wedge t) - \eta(a^n_{i}\wedge t)\Big).
\end{equation}
Observe the right-hand side of (\ref{pathwiseint}) only depends on $(\rho,\eta)$ over the interval $[0,t]$. We observe that the above sum is finite for each $\rho,\eta\in D([0,t];\mathbb{R})$. We then define $\mathcal{I}_t:D([0,t];\mathbb{R})\times D([0,t];\mathbb{R}) \rightarrow\mathbb{R}$ as

$$\mathcal{I}_t\big(\rho_t, \eta_t\big):=\lim_{n\rightarrow+\infty}\mathcal{I}_{t,n}\big(\rho_t, \eta_t\big)$$
if the limit exists and we set $\mathcal{I}_t\big(\rho_t, \eta_t\big):=0$ otherwise. Let

$$E = \Big\{(\rho,\eta); \lim_{n\rightarrow+\infty}\mathcal{I}_{t,n}\big(\rho_t, \eta_t\big)~\text{exists for every}~t\in [0,T]\Big\}$$
From Karandikar \cite{karandikar}, we know that

$$(Z(B),B)\in E~a.s$$
as long as $Z(B)$ has c\`adl\`ag paths. We then define

\begin{equation}\label{pathwiseMARTINGALE}
M_t(c_t): = \mathcal{I}_t\Big( Z_t(c_t), c_t \Big)
\end{equation}
where $Z(c)$ has continuous paths for every $c\in D([0,T];\mathbb{R})$.

\begin{lemma}\label{martingaleREP}
If $(Z(c),c)\in E$ and $s\mapsto Z_s(c_s)$ is continuous, then $\big(Z(\mathbf{t}(c,x)), \textbf{t}(c,x)\big)\in E$ and

$$\mathcal{I}_t\big(Z_t(\textbf{t}(c_t,x)),\textbf{t}(c_t,x)\big) = \mathcal{I}_t \big( Z_t(c_t),c_t \big) + Z_t(c_t)\big(x- c(t)\big)$$
foe every $t\in [0,T]$ and $x\in \mathbb{R}$. In particular, if $s\mapsto Z_s(c_s)$ is continuous for every $c\in D([0,T];\mathbb{R})$, then $\lim_{n\rightarrow+\infty}\mathcal{I}_{t,n}\big( Z_t(c_t), c_t \big)$ exists, if and only if, $\lim_{n\rightarrow+\infty} \mathcal{I}_{t,n}\big(Z_t(\textbf{t}(c_t,x)), \textbf{t}(c_t,x)\big)$ exists for every $x\in \mathbb{R}$.
\end{lemma}
\begin{proof}
Assume $(Z(c),c)\in E$, $s\mapsto Z_s(c_s)$ is continuous and we fix $x\in \mathbb{R}$ and $t$. Let us denote

$$ Z^n_s(c_s):= \sum_{j=0}^\infty Z_{a^n_j}(c_{a^n_j})\mathds{1}_{\{a^n_j \le s < a^n_{j+1}  \}}$$
where the sequence $a^n_j$ is computed based on the path of $\mapsto Z_s(c_s)$. We observe

$$\mathcal{I}_{t,n}\big(Z_t(c_t), c_t\big)=\mathcal{I}_t\big(Z^n_t(c_t),c_t\big) = Z_0(c_0)c(0) +\sum_{j=0}^{p-1} Z_{a^n_j}(c_{a^n_j}) \big( c(a^n_{j+1})  - c(a^n_j)\big)$$
$$+ Z_{a^n_p}(c_{a^n_p}) \big( c(t) - c(a^n_p)  \big)~\text{on}~\{a^n_p \le t < a^n_{p+1}\}.$$
Of course, by the very definition

$$\lim_{n\rightarrow +\infty}\mathcal{I}_{t,n}\big(Z_t(c_t),c_t\big) = \mathcal{I}_t\big(Z_t(c_t),c_t\big).$$
Moreover,

$$\mathcal{I}_{t,n}\big( Z_t(\textbf{t}(c_t,x)), \textbf{t}(c_t,x) \big) = \mathcal{I}_{t,n}\big( Z_t(c_t), \textbf{t}(c_t,x) \big)$$
and

$$\mathcal{I}_{t,n}\big( Z_t(\textbf{t}(c_t,x)),  \textbf{t}(c_t,x) \big) - \mathcal{I}_{t,n}\big( Z_t(c_t), c_t \big) = Z_{a^n_p}(c_{a^n_p})\big( (x-c(a^n_p)) - (c(t) - c(a^n_p)) \big)~\text{on}~\{a^n_p \le t < a^n_{p+1}\}.$$
By the left-continuity of $Z(c)$, we have

$$\lim_{n\rightarrow+\infty}\Big(\mathcal{I}_{t,n}\big( Z_t(\textbf{t}(c_t,x)),  \textbf{t}(c_t,x) \big) - \mathcal{I}_{t,n}\big( Z_t(c_t), c_t \big)\Big) = Z_t(c_t)\big( x - c(t)\big).$$
This shows that

$$\lim_{n\rightarrow+\infty} \mathcal{I}_{t,n}\big(Z_t(\textbf{t}(c_t,x)), \textbf{t}(c_t,x)\big) = \mathcal{I}_t\big(Z_t(c_t), c_t\big) + Z_t(c_t) \big(x-c(t)\big).$$
This concludes the proof.
\end{proof}

Lemma \ref{martingaleREP} yields $x\mapsto M_t(\textbf{t}(c,x))$ is $C^1(\mathbb{R})$ and B2 is fulfilled, where

$$\nabla_x M_t(\textbf{t}(c,x)) = Z_t(c_t)$$
for each $(x,t)\in \mathbb{R}\times [0,T]$ and $c\in D([0,T];\mathbb{R})$. Moreover,
$$M_t(\textbf{t}(B_t,B(t))) = \int_0^t Z_s(B_s)dB(s)~a.s, 0\le t\le T.$$
The candidate to be a functional imbedded structure is
$$\mathbf{M}^k(t) = \sum_{n=0}^\infty\int_0^{T^k_n}Z_s(A^k_s)dA^k(s)\mathds{1}_{\{T^k_n\le t < T^k_{n+1}\}}; 0\le t\le T.$$
If we assume $\nabla_x M$ satisfies B3-B4, then $\lim_{k\rightarrow +\infty}\mathbf{M}^k = M(B)$ weakly in $\mathbf{B}^2(\mathbb{F})$. Assumption B5 clearly holds true because

$$M_{t+h}(c_{t,h}) - M_t(c_t) = 0$$
for every $h>0$ and $t$. Assumption B6 holds true because

\begin{eqnarray*}
\sup_{k\ge 1}\epsilon^{-2}_k\mathbb{E}|\Delta \mathbf{M}^k(T^k_{N^k(T)})|^2\big|T^k_{N^k(T)}-T\big| &=& \sup_{k\ge 1}\epsilon^{-2}_k\mathbb{E}\Bigg|\int_{T^k_{N^k(T)-1}}^{N^k(T)} Z_s(A^k_s)dA^k(s)\Bigg|^2\big|T^k_{N^k(T)}-T\big|\\
&\le&\sup_{k\ge 1}\mathbb{E}\Big|Z_{T^k_{N^k(T)-1}}\big(A^k_{T^k_{N^k(T)-1}}\big)\Big|^2|T^k_{N^k(T)}-T|< \infty
\end{eqnarray*}
due to the uniform integrability assumption in B4. Finally, if we define

\begin{equation}\label{roughGAS}
\mathbf{F}^k(t) = \mathbf{M}^k(t) + \mathbf{J}^k(t); 0\le t \le T,
\end{equation}
then by applying Proposition \ref{roughcase} to the imbedded discrete structure (\ref{roughGAS}), we conclude $X$ given by (\ref{roughX}) is weakly differentiable. By applying Proposition \ref{ulocaltime}, we arrive at the following result.
\begin{theorem}\label{roughcasePROP}
Let $X$ be the Wiener functional given by (\ref{roughX}) with a functional representation

$$X(t) = J_t(B_t) + M(B_t)$$
given by (\ref{pathwiseDRIFT}) and (\ref{pathwiseMARTINGALE}). Assume $Y$ satisfies (\ref{B7}), (\ref{B8}), $\nabla_x M = Z$ satisfies B3-B4 and $Z(c)$ has continuous paths for every $c\in D([0,T];\mathbb{R})$. Then, $X$ is weakly differentiable,

$$\mathcal{D}X = Z(B)$$
and
$$\int_0^\cdot Y(s)dg(s) = \lim_{k\rightarrow \infty}\Bigg[\int_0^\cdot \mathbb{D}^{\mathcal{Y},k,h}X(s)ds - \frac{1}{2}\int_0^\cdot\int_{-\infty}^{+\infty}\nabla^ {\mathcal{Y},k,v}X(s,x)d_{(s,x)}\mathbb{L}^{k,x}(s)\Bigg]$$
weakly in $\mathbf{B}^2(\mathbb{F})$ for every stable imbedded discrete structure $\mathcal{Y}$ associated with $X$.
\end{theorem}

\section{Weak differentiability and local times}\label{Youngsection}
In this section, we investigate in detail the convergence (\ref{v}) which characterizes the non-martingale component of a given $X\in \mathcal{W}(\mathbb{F})$ in Proposition \ref{ulocaltime}. The existence of

$$
\lim_{k\rightarrow \infty} \int_0^\cdot \mathbb{D}^{\mathcal{Y},k,h}X(s)ds
$$
requires strong regularity in typical examples path-dependent Wiener functionals and it is similar in nature to the existence of the horizontal derivative $\nabla^h F$ in the sense of \cite{dupire,cont}. The most interesting and non-trivial object to be analyzed is

\begin{equation}\label{FUNDA_LIMIT}
\int_0^\cdot\int_{-\infty}^{+\infty}\nabla^ {\mathcal{Y},k,v}X(s,x)d_{(s,x)}\mathbb{L}^{k,x}(s).
\end{equation}
In this section, we devote our attention to the limit (\ref{FUNDA_LIMIT}). In particular,  we aim to show that for a given $X\in \mathcal{W}(\mathbb{F})$, the existence of the limit (\ref{FUNDA_LIMIT}) is connected with the two-dimensional $(p,q)$-variation regularity of the Wiener functional $X$. This connection will be explored by means of the 2D-Young integral \cite{young1}. The precise identification of the limit (\ref{FUNDA_LIMIT}) is not a trivial task. In terms of Young integration theory (see e.g~\cite{young,young1}), the first obstacle towards the asymptotic limit~(\ref{FUNDA_LIMIT}) is to handle two-dimensional variation of $(\mathbb{L}^{k,x}(t), \ell^x(t))$ over $\mathbb{R}\times [0,T]$ simultaneously with the infinitesimal behavior of the past of a Brownian path composed with the functional $F$. In order to prove convergence of the local-time integrals in (\ref{FUNDA_LIMIT}), sharp maximal estimates on the number of crossings and the 2D Young integrals play key roles (see \cite{OS,OS1}) for the obtention of the limit. In this article, we focus our attention to Young's approach and the analysis of other possible characterizations depending on the roughness of $X$ will be discussed elsewhere.

The analysis of the limit (\ref{FUNDA_LIMIT}) without imposing pathwise regularity conditions on a given explicit functional $F$ realizing (\ref{functionalrepresentation}) is very challenging and we postpone this question to a further investigation. In the remainder of this section, we then fix a functional imbedded discrete structure $\mathcal{F} = \big((\textbf{F}^k)_{k\ge 1},\mathscr{D}\}$ of the form~(\ref{funcSTRUCTURE}) where $F$ realizes (\ref{functionalrepresentation}) and $\mathscr{D}$ is driven by the sequence $\epsilon_k = 2^{-k}$. In the sequel, $d_{(s,x)}$ and $d_x$ are computed in the sense of a 2D and 1D Young integral, respectively. See the seminal L.C Young's articles~\cite{young,young1} for further details. In this section, it will be more convenient to work with processes defined on the whole period $[0,+\infty)$ but keeping in mind that we are just interested on the bounded interval $[0,T]$. For this reason, all processes are assumed to be defined over $[0,+\infty)$ but they vanish after the finite time $T$.

In this section, it is convenient to work with the following ``clock" modification of $\mathbb{L}^k$ as follows

$$
L^{k,x}(t): = \sum_{j\in\mathbb{Z}} l^{k,j2^{-k}}(t)1\!\!1_{S^k_j}(x),
$$
where

$$
\l^{k,x}(t):= \frac{1}{2^ {-k}}\int_0^t 1\!\!1_{\{|A^ k(s-) - x| < 2^{-k} \}}d[A^k,A^k](s);~k\ge 1,~x\in \mathbb{R},~ 0\le t < +\infty.
$$

For each $k\ge 1$ and $x\in \mathbb{R}$, let $j_k(x)$ be the unique integer such that $(j_k(x)-1)2^{-k} < x \le j_k(x)2^{-k}$ and $N^k(t) = \max \{n; T^k_n \le t\}$ is the length of the embedded random walk until time $t$. By the very definition, $L^{k,x}(t) = l^{k,j_k(x)2^{-k}}(t);~(x,t)\in \mathbb{R}\times [0,\infty).$ More precisely      ,

\begin{eqnarray*}
L^{k,x}(t)&=& 2^{-k}\#\ \big\{n \in \{1, \ldots, N^k(t)-1\}; A^k(T^k_{n}) =j_k(x)2^{-k}\big\} \\
& &\\
&=&2^{-k}\Big(u(j_k(x)2^{-k},k,t) + d(j_k(x)2^{-k},k,t)\Big);~(x,t)\in \mathbb{R}\times [0,+\infty)
\end{eqnarray*}
where
$$u(j_k(x)2^{-k},k,t):= \#\ \big\{n \in \{1, \ldots, N^k(t)-1\}; A^k(T^k_{n-1}) =(j_k(x)-1)2^{-k}, A^k(T^k_{n}) =j_k(x)2^{-k}\big\};$$

$$d(j_k(x)2^{-k},k,t):= \#\ \big\{n \in \{1, \ldots, N^k(t)-1\}; A^k(T^k_{n-1}) =(j_k(x)+1)2^{-k}, A^k(T^k_{n}) =j_k(x)2^{-k}\big\};$$
for~$x\in \mathbb{R}, k\ge 1, 0\le t < \infty.$

In the sequel, we denote $\{\ell^x(t); (x,t)\in \mathbb{R}\times [0,+\infty)\}$ as the standard Brownian local-time, i.e., it is the unique jointly continuous random field which realizes

$$\int_{\mathbb{R}}f(x)\ell^x(t)dx = \int_0^tf(B(s))ds\quad \forall t > 0~\text{and measurable}~f:\mathbb{R}\rightarrow \mathbb{R}.$$

We now state a result which plays a key role for establishing the existence of limit (\ref{FUNDA_LIMIT}). In the remainder of this section, we set $I_m:= [-2^m,2^m]$ for a positive integer $m\ge 1$.

\begin{lemma}\label{lemmaLk}
For each $m\ge 1$, the following properties hold:

\

\noindent (i) $L^{k,x}(t)\rightarrow \ell^x(t)\quad\text{a.s uniformly in}~(x,t)\in I_m\times [0,T]$ as $k\rightarrow \infty.$

\

\noindent (ii) $\sup_{k\ge 1}\mathbb{E}\sup_{x\in I_m}\|L^{k,x}(\cdot)\|^p_{[0,T]; 1}< \infty$ and $\sup_{k\ge 1}\sup_{x\in I_m}\|L^{k,x}(\cdot)\|_{[0,T]; 1}< \infty~a.s$ for every $p\ge 1$.


\

\noindent (iii) $\sup_{k\ge 1}\mathbb{E}\sup_{t\in [0,T]}\| L^k(t)\|^{2+\delta}_{I_m;2+\delta}< \infty$ and $\sup_{k\ge 1}\sup_{t\in [0,T]}\| L^k(t)\|^{2+\delta}_{I_m;2+\delta}< \infty~a.s$~for every $\delta>0$.
\end{lemma}
\begin{proof}
The component $2^{-k}u(j_k(x)2^{-k},k,t)\rightarrow \frac{1}{2}\ell^x(t)$ a.s uniformly over $(x,t)\in I_m\times [0,T]$ due to the classical Th. 4.1 in~\cite{knight}. By writing $2^{-k}d(j_k(x)2^{-k},k,t) - \frac{1}{2}\ell^x(t) = 2^{-k}d(j_k(x)2^{-k},k,t) - 2^{-k}d((j_k(x)-1)2^{-k},k,t) + 2^{-k}d((j_k(x)-1)2^{-k},k,t) - \frac{1}{2}\ell^x(t)$ and using Lemma 6.23 in~\cite{peres}, we conclude that item (i) holds.
The proof of (ii) is quite simple because $ L^{k,x}(\cdot)$ has increasing paths a.s for each $k\ge 1$ and $x\in I_m$. In fact, Th.1 in~\cite{Barlow} yields

$$
\mathbb{E}\sup_{k\ge 1}\sup_{x\in I_m}\Bigg|\sup_{\Pi}\sum_{t_i\in \Pi} |L^{k,x}(t_{i}) - L^{k,x}(t_{i-1})|\Bigg|^p\le \mathbb{E}\sup_{r\ge 1}\sup_{x\in I_m}|L^{r,x}(T)|^p<\infty,
$$
for every $p\ge 1$. This shows that (ii) holds. For the proof of item (iii), it will be sufficient to check only for the upcrossing component of $L^k$. Let us fix $M,\delta>0$. For a given partition $\Pi=\{x_i\}_{i=0}^N$ of the interval $I_m$, let us define the following subset $\Lambda(\Pi,k): = \{x_i\in \Pi; (j_k(x_i)-j_k(x_{i-1}))2^{-k} >0 \}$. We notice that $\#\ \Lambda(\Pi,k)\le 22^{k+m}$ for every partition $\Pi$ of $I_m$. To shorten notation, let us write $U_p^k(t,x): = 2^{-k}u^k(j_k(x)2^{-k},t);~(x,t)\in I_m \times [0,T]$. We readily see that

\begin{equation}\label{es1}
\sum_{x_i\in \Pi}|U_p^{k}(t,x_i)  - U_p^{k}(t,x_{i-1})|^{2+\delta}\le \sum_{x_i\in \Lambda(\Pi,k)}|U_p^{k}(t,x_i)  - U_p^{k}(t,x_{i-1})|^{2+\delta}~a.s,
\end{equation}
for every partition $\Pi$ of $I_m$. By writing $|U_p^{k}(t,x_i)  - U_p^{k}(t,x_{i-1})| = |U_p^{k}(t,x_i) - 1/2\ell^{x_i}(t) + 1/2\ell^{x_i}(t) - 1/2\ell^{x_{i-1}}(t) + 1/2\ell^{x_{i-1}}(t) - U_p^{k}(t,x_{i-1})|;~x_i\in \Lambda(\Pi,k)$ and applying the standard inequality $|\alpha -\beta|^{2+\delta}\le 2^{1+\delta}\{|\alpha|^{2+\delta} + |\beta|^{2+\delta}\};~\alpha,\beta\in \mathbb{R}$, the bound~(\ref{es1}) yields

\begin{eqnarray}
\nonumber \sup_{0\le t\le T}\|U_p^k(t)\|^{2+\delta}_{I_m;2+\delta} &\le& 2^{1+\delta}\sup_{0\le t\le T}\|1/2\ell^x(t)\|^{2+\delta}_{I_m;2+\delta} + \nonumber2^{1+\delta}\sup_{0\le t\le T}\sup_{\Pi}\sum_{x_i\in \Lambda(\Pi,k)}|U_p^k(t,x_i) - 1/2\ell^{x_i}(t)|^{2+\delta}\\
\nonumber& &\\
\label{l2}&+& 2^{1+\delta}\sup_{0\le t\le T}\sup_{\Pi}\sum_{x_i\in \Lambda(\Pi,k)}|U_p^k(t,x_{i-1}) - 1/2\ell^{x_{i-1}}(t)|^{2+\delta}~a.s~\text{for every}~k\ge 1.
\end{eqnarray}
An inspection in the proof of Lemma 2.1 in~\cite{feng} yields $\sup_{0\le t\le T}\|\ell^\cdot(t)\|^{2+\delta}_{I_m;2+\delta}<\infty$. By applying Th. 1.4 and Remark 1.7.1 in~\cite{kho}, we get for every $k$ larger than a positive random number, the following bound
\small
\begin{eqnarray*}
\sup_{0\le t\le T}\sup_{\Pi}\sum_{x_i\in \Lambda(\Pi,k)}|U^k_p(t,x_i) - 1/2\ell^{x_i}(t)|^{2+\delta}&\le& \sup_{0\le t\le T}\sup_{x\in I_m}|U^k_p(x,t) - 1/2\ell^x(t)|^{2+\delta}2 2^{k+m}\\
& &\\
&\le& 2 2^{k+m}2^{-\frac{k}{2}(2+\delta)}k^{\frac{2+\delta}{2}}(log(2))^{\frac{2++\delta}{2}} \Big(M + \sup_{0\le t\le T}\sqrt{2}\sqrt{\ell^*(t)} \Big)^{2+\delta}\\
& &\\
&\le& (C + C\ell^*(T)^{1+\frac{\delta}{2}}) 2^{-\frac{k\delta}{2}}k^{1+\frac{\delta}{2}},
\end{eqnarray*}
\normalsize
where $C$ is a positive constant which only depends on $(M,n)$ and $\ell^*(t):=\sup_{x\in I_m}\ell^x(t);~0\le t\le T$. The last term in~(\ref{l2}) can be treated similarly. By using the fact that $2^{-\frac{k\delta}{2}}k^{1+\frac{\delta}{2}} = O(1)$ and $\ell^*(T) < \infty$ a.s, we conclude that $\sup_{k\ge 1}\sup_{t\in [0,T]}\| L^k(t)\|^{2+\delta}_{I_m;2+\delta}< \infty~a.s$~for every $\delta>0$. It remains to check the $L^{2+\delta}$-bound in (iii) which is more delicate than the related almost sure bound. We refer the reader to Corollary 2.1 in~\cite{OS1} for a detailed proof of this bound.
\end{proof}

In order to get an explicit limit for the space-time local time integrals in terms of a 2D Young integral, we need to assume some pathwise regularity conditions. More precisely, let $\mathcal{C}^1$ be the set of non-anticipative functionals $\{F_t; 0\le t\le  T\}$ such that $x\mapsto F_s(\textbf{t}(c_s,x))$ is $C^1(\mathbb{R})$ for each $c\in D([0,+\infty);\mathbb{R})$, $s\ge 0$.

\

\noindent \textbf{Assumption L1}: For every c\`adl\`ag function $c$, the map $c\mapsto \nabla_xF_s(\textbf{t}(c_s,x))$ is continuous uniformly in the time variable $s$, i.e., for every compact set $K\subset \mathbb{R}$ and $\varepsilon > 0$ there exists $\delta>0$ such that $\sup_{0\le s\le  T}|c(s) - d(s)|< \delta \Longrightarrow \sup_{0\le s\le  T}\sup_{x\in K}|\nabla_x F_s(\textbf{t}(c_s,x)) - \nabla_xF_s(\textbf{t}(d_s,x))| < \varepsilon$.

\

\begin{remark}
One should notice that this implies, in particular, that if $c^k$ is a sequence of c\`adl\`ag functions such that $\sup_{0\le t\le T} |c^k(t)-d(t)| \to 0$ as $k\to \infty$, then $\sup_{0\le s\le  T}\sup_{x\in K}|\nabla_x F_s(\textbf{t}(c^k_s,x)) - \nabla_xF_s(\textbf{t}(d_s,x))|\to 0$, as $k\to\infty$ for every compact set $K\subset \mathbb{R}$.
\end{remark}

The following lemma is straightforward in view of the definitions. We left the details of the proof to the reader.

\begin{lemma}\label{lemmayoungnabla}
If $F\in \mathcal{C}^1$ satisfies Assumption $\textbf{(L1)}$, then $\nabla^{\mathcal{F},v,k}F_s(B_s,x)\rightarrow \nabla_xF_s(\textbf{t}(B_s,x))$ a.s uniformly in $(x,s)\in I_m\times [0,T]$ for every $m\ge 1$.
\end{lemma}

Let us now introduce additional hypotheses in order to work with 2D-Young integrals. We refer the reader to Young~\cite{young1} for further background. Let us denote 
$$\Delta_ih(t_i,x_j): = h(t_i,x_j)-h(t_{i-1},x_j)$$ 
as the first difference operator acting on the variable $t$ of a given function $h:[0,T]\times[-L,L]\rightarrow \mathbb{R}$.

\

\noindent \textbf{Assumption~L2.1}: Assume for every $L>0$, there exists a positive constant $M$ such that

\begin{equation}\label{l2.1}
|\Delta_i\Delta_j\nabla_xF_{t_i}(\textbf{t}(B_{t_i},x_j))|\le M|t_i-t_{i-1}|^{\frac{1}{q_1}}|x_j-x_{j-1}|^{\frac{1}{q_2}}~a.s
\end{equation}
for every partition $\Pi= \{t_i\}_{i=0}^N\times \{x_j\}_{j=0}^{N^{'}}$ of $[0,T]\times [-L,L]$, where $q_1,q_2>1$. There exists $\alpha\in (0,1),\delta >0, p \ge 1$ such that $\min\{\alpha + \frac{1}{q_1}, \frac{1-\alpha}{2+\delta} + \frac{1}{q_2}\} > 1$, $\frac{1}{p} + \frac{1}{2+\delta}>1$ and

$$
\sup_{0\le t\le T}\|\nabla_xF_t(\textbf{t}(B_t,\cdot))\|_{[-L,L];p}\in L^{\infty}(\mathbb{P}),
$$
for every $L> 0$.

\

\noindent \textbf{Assumption~L2.2}: In addition to assumption \textbf{(L2.1)}, let us assume  $\forall L >0$, there exists $M>0$ such that

\begin{equation}\label{l2.3}
\sup_{k\ge 1} |\Delta_i\Delta_j\nabla^{\mathcal{F},k,v}F_{t_i}(B_{t_i},x_j)|\le M |t_i-t_{i-1}|^{\frac{1}{q_1}}|x_j-x_{j-1}|^{\frac{1}{q_2}} ~a.s.
\end{equation}
for every partition $\Pi= \{t_i\}_{i=0}^N\times \{x_j\}_{j=0}^{N^{'}}$ of $[0,T]\times [-L,L]$, and

\begin{equation}\label{l2.4}
\sup_{k\ge 1}\sup_{0\le t\le T}\| \nabla^{\mathcal{F},k,v}F_t(B_t,\cdot)\|_{[-L,L],p} \in L^\infty(\mathbb{P}).
\end{equation}

\begin{remark}\label{localremark}
In the language of rough path theory, assumption~(\ref{l2.1}) precisely says that if $q=q_1=q_2$ then $\nabla_x F_t(\textbf{t}(B_t,x))$ admits a 2D-control $\omega([t_1,t_2]\times [x_1,x_2]) = |t_1-t_2|^{\frac{1}{q}}|x_1-x_2|^{\frac{1}{q}}$ so that~(\ref{l2.1}) trivially implies that $(t,x)\mapsto \nabla_x F_t(\textbf{t}(B_t,x))$ has $q$-joint variation in the sense of~\cite{friz}. Assumptions \textbf{(L2.1-L2.2)} are sufficiently rich to accommodate a large class of examples.
\end{remark}

We are now in position to state our first approximation result which makes heavily use of the pathwise 2D Young integral in the context of the so-called $(p,q)$-bivariation rather than joint variation. We refer the reader to the seminal Young article~\cite{young1}(section 6) for further details.

\begin{proposition}\label{convyoung}
Let $F\in\mathcal{C}^1$ satisfy assumptions (\textbf{L1,L2.1,~L2.2}). Then, for every $t\in [0,T]$ and $m\ge 1$, we have

\begin{equation}\label{f3}
\lim_{k\rightarrow \infty}\int_0^t \int_{-2^m}^{2^m}\nabla^{\mathcal{F},k,v}F_r(B_r,x)d_{(r,x)} L^{k,x}(r)=\int_0^t\int_{-2^m}^{2^m}\nabla_x  F_r(\textbf{t}(B_r,x))d_{(r,x)}\ell^{x}(r)\quad \text{a.s},
\end{equation}
where the right-hand side of~(\ref{f3}) is interpreted as the pathwise 2D Young integral.
\end{proposition}
\begin{proof}
Let us fix $t\in [0,T]$ and a positive integer $m$. Let $\{0=s_0 < s_1<\ldots < s_p =t\}$ and $\{-2^m = x_0 < x_1 < \ldots < x_l=2^m\}$ be partitions of $[0,t]$ and $[-2^m,2^m]$, respectively. Similar to identity~(4.5) in~\cite{feng}, we shall write

\begin{eqnarray*}
\sum_{i=0}^{l-1}\sum_{j=0}^{p-1}\nabla_xF_{s_j}(\textbf{t}(B_{s_j},x_i))\Delta_i\Delta_j \ell^{x_{i+1}}(s_j+1)&=&\sum_{i=1}^{l}\sum_{j=1}^{p}\ell^{x_{i}}(s_j)  \Delta_i\Delta_j\nabla_x F_{s_j}(\textbf{t}(B_{s_j},x_i))\\
& &\\
&-& \sum_{i=1}^l\ell^{x_i}(t)\Delta_i F_t(\textbf{t}(B_t,x_i))
\end{eqnarray*}
From Lemmas 2.1,~2.2 in~\cite{feng}, we know that $\{\ell^x(s); 0\le s\le T; x \in I_m\}$ has $(1,2+\delta)$-bivariations a.s for every $\delta > 0$ (See Young~\cite{young1}, p.~583 for further details). Then by applying Th. 6.3 in~\cite{young1} and the one dimensional existence theorem in~\cite{young} together with \textbf{L2.1}, we have

\begin{eqnarray*}
\int_0^t\int_{-2^m}^{2^m}\nabla_x F_s(\textbf{t}(B_s,x))d_{(s,x)}\ell^x(s)&=& \int_0^t\int_{-2^m}^{2^m}\ell^x(s)d_{(s,x)}\nabla_xF_s(\textbf{t}(B_s,x))\\
& &\\
&-& \int_{-2^m}^{2^m}\ell^x(t)d_x\nabla_xF_t(\textbf{t}(B_t,x))~a.s.
\end{eqnarray*}
We do the same argument to write

\begin{eqnarray}
\nonumber\int_0^t\int_{-2^m}^{2^m}\nabla^{\mathcal{F},k,v}F_s(B_s,x)d_{(s,x)}L^{k,x}(s) &=& \nonumber\int_0^t\int_{-2^m}^{2^m}L^{k,x}(s)d_{(s,x)}\nabla^{\mathcal{F},k,v}F_s(B_s,x)\\
\label{splitk}& &\\
\nonumber&-& \int_{-2^m}^{2^m}L^{k,x}(t)d_x\nabla^{\mathcal{F},k,v}F_t(B_t,x);a.s~0\le t\le T; k\ge 1.
\end{eqnarray}
Now we apply $\textbf{(L1-L2.1-L2.2)}$, Lemma~\ref{lemmaLk} together with Th. 6.3, 6.4 in~\cite{young1} and the term by term integration theorem in~\cite{young} to conclude that

$$\lim_{k\rightarrow \infty}\int_0^t\int_{-2^m}^{2^m}L^{k,x}(s)d_{(s,x)}\nabla^{\mathcal{F},k,v}F_s(B_s,x) = \int_0^t\int_{-2^m}^{2^m}\nabla_x F_s(\textbf{t}(B_s,x))d_{(s,x)}\ell^x(s)~a.s
$$

$$
\lim_{k\rightarrow \infty}\int_{-2^m}^{2^m}L^{k,x}(t)d_x\nabla^{\mathcal{F},k,v}F_t(B_t,x) = \int_{-2^m}^{2^m}\ell^x(t)d_x\nabla_xF_t(\textbf{t}(B_t,x))~a.s$$
up to some vanishing conditions on the boundaries $t=0$ and $x=-2^m$. They clearly vanish for $t=0$. For $x=-2^m$ we have to work a little. In fact, we will enlarge our domain in $x$, from $[-2^m,2^m]$ to $[-2^m-1,2^m]$, and define for all functions with $-2^m-1\le x < -2^m$, the value $0$, that is, for the functions $L^{k,x}(s), \ell^x(s), \nabla_x F_s(\textbf{t}(B_s,x))$ and $\nabla^{\mathcal{F},v,k}F_s(B_s,x)$ we put the value $0$, whenever $x\in [-2^m-1,-2^m)$. Then, it is easy to see, that all the conclusions of Lemmas \ref{lemmaLk} and \ref{lemmayoungnabla} still hold true, and in this case, $L^{k,-2^m-1}(s)=0$ and $\ell^{-2^m-1}(s)=0$ for all $s$. Thus, we can apply Theorems~6.3 and 6.4 in~\cite{young1} on the interval $[0,t]\times [-2^m-1,2^m]$. This concludes the proof.
\end{proof}
Next, we need to translate convergence~(\ref{f3}) into $L^1$ convergence. The maximal inequality derived by~\cite{OS} (see Th. 1.3 and Corollary 1.1) based on bivariations plays a crucial role in the proof of the following lemma. See Remark~\ref{localremark}.


\begin{proposition}\label{ltlema}
Assume that $F\in\mathcal{C}^1$ satisfies assumptions~\textbf{(L1-L2.1-L2.2)}. Then, for every non-negative random variable $J$, we have

\begin{equation}\label{ltc}
\lim_{k\rightarrow \infty}\int_0^J\int_{-2^m}^{2^m}\nabla^{\mathcal{F},k,v}F_r(B_r,x)d_{(r,x)}L^{k,x}(r) = \int_0^J\int_{-2^m}^{2^m}\nabla_x F_r(\textbf{t}(B_r,x))d_{(r,x)}\ell^x(t)
\end{equation}
strongly in $L^1(\mathbb{P})$ for every $m\ge 1$.
\end{proposition}
\begin{proof}
Let us fix a positive random variable $J$. In view of~(\ref{splitk}) and Proposition~\ref{convyoung}, it is sufficient to check that $|\int_0^J\int_{-2^m}^{2^m}L^{k,x}(s)d_{(s,x)}\nabla^{\mathcal{F},k,v}F_s(B_s,x)| + |\int_{-2^m}^{2^m}L^{k,x}(J)d_{x}\nabla^{\mathcal{F},k,v}F_J(B_J,x)|=: I^k_1 + I^k_2;k\ge 1,$ is uniformly integrable. By the classical Young inequality based on simple functions, the following bound holds

$$I^k_2\le C_1 \Big(L^{k,-2^m}(J) + \|L^{k}(J)\|_{[-2^m,2^m];2+\delta}\Big)~a.s$$
where $C_1$ is a constant depending on $\delta$ and~(\ref{l2.4}). From Th. 1 in~\cite{Barlow} and item (ii) in Lemma~\ref{lemmaLk}, we conclude that $\{I^k_2; k\ge 1\}$ is uniformly integrable.

A direct application of Cor. 1.1 in~\cite{OS} together with~\textbf{(L2.2)} yield

\begin{eqnarray}
\nonumber\Big|\int_0^J\int_{-2^m}^{2^m}L^{k,x}(s)d_{(s,x)}\nabla^{\mathcal{F},k,v}F_s(B_s,x)\Big| &\le& K_0L^{k,2^m}(T) +  K\|L^{k}\|^\alpha_{time,1}\|L^{k}\|^{1-\alpha}_{space,2+\delta}\\
\nonumber& &\\
\label{st11}&+& K_1\|L^k\|_{time,1} + K_2\|L^k\|_{space,2+\delta}
\end{eqnarray}

where
$$\|L^{k}\|_{space,2+\delta}:=\sup_{(t,s)\in [0,T]^2}\|L^k(t) - L^k(s)\|_{I_m;2+\delta},$$

$$\|L^{k}\|_{time,1}:=\sup_{(x,y)\in I_m\times I_m}\|L^{k,x} - L^{k,y}\|_{[0,T];1}.$$
Here $K_0$ is a constant which comes from assumption~(\ref{l2.3}) and $K,K_1,K_2$ are positive constants which only depend on the constants of assumption $\textbf{(L2.1, L2.2)}$ namely $\alpha,q_1,q_2,\delta, T \wedge J, m$. From Lemma~\ref{lemmaLk}, we have $\sup_{k\ge 1}\mathbb{E}\|L^{k}\|^{2+\delta}_{space,2+\delta} < \infty$, $\mathbb{E}\sup_{k\ge 1}\|L^{k}\|^q_{time,1} < \infty$ for every $q\ge 1$. Again Th.1 in~\cite{Barlow} yields the uniform integrability of $\{L^{k,2^m}(T);k\ge 1\}$. So we only need to check uniform integrability of $\{\|L^{k}\|^\alpha_{time,1}\|L^{k}\|^{1-\alpha}_{space,2+\delta}; k\ge 1\}$ in~(\ref{st11}). For $\beta > 1$, we apply H\"{o}lder inequality to get

$$\mathbb{E}\|L^{k}\|^{\beta\alpha}_{time,1}\|L^{k}\|^{(1-\alpha)\beta}_{space,2+\delta}\le \big(\sup_{r\ge 1}\mathbb{E}\|L^r\|_{space;2+\delta}\big)^{1/p}\big(\sup_{r\ge 1}\mathbb{E}\|L^r\|^{\alpha\beta q}_{time;1}\big)^{1/q};~k\ge 1$$
where $p=\frac{1}{(1-\alpha)\beta} >1,~q = \frac{p}{p-1}=\frac{1}{1-(1-\alpha)\beta}$ with $\alpha \in (0,1)$. By applying Lemma~\ref{lemmaLk}, we conclude that $\{I^k_1; k\ge 1\}$ is uniformly integrable. Finally, Proposition~\ref{convyoung} allows us to conclude that (\ref{ltc}) holds.


\end{proof}

\begin{lemma}\label{fkst}
For every $\mathbb{F}$-stopping time $S$, there exists a sequence of non-negative random variables $\{J^k ; k\ge  1 \}$ such that
$J^k$ is an $\mathbb{F}^k$-predictable stopping time for each $k \ge 1$ and $\lim_{k\rightarrow \infty} J^k = S$ a.s.
\end{lemma}
\begin{proof}
We can repeat the same steps of the proof of Lemma 2.2 in~\cite{LEAO_OHASHI2013}, to get for a given $\mathbb{F}$-predictable set $O$, a sequence $O^k$ of $\mathbb{F}^k$-predictable set such that $\mathbb{P}[\pi(O) - \pi(O^k)]\rightarrow 0$ as $k\rightarrow \infty$, where $\pi$ is the usual projection of $\mathbb{R}_+\times \Omega$ onto $\Omega$. Apply the usual Section Theorem~(see e.g~\cite{he}) on each $O^k$ and on the graph $[[S,S]] = \{(\omega,t);S(\omega=t)\}$ as in [Lemma 3.3;~\cite{LEAO_OHASHI2013}] to get a sequence $\bar{J}^k$ of $\mathbb{F}^k$-predictable stopping times such that $\lim_{k\rightarrow \infty}\bar{J}_k=J~a.s$ on $\{J < \infty\}$. By defining, $J^k:=\bar{J}^k$ on $\{J< \infty\}$ and $J^k:=+\infty$ on $\{J=+\infty\}$, we get the desired sequence.
\end{proof}

Now we are in position to state the main result of this section.

\begin{theorem}\label{youngTh}
Let $F\in \mathcal{C}^1$ be a non-anticipative functional such that $F(B)\in \mathcal{W}(\mathbb{F})$ equipped with a stable imbedded discrete structure $\mathcal{F} = \big((\mathbf{F}^k)_{k\ge 1},\mathscr{D}\big)$. Assume that $\mathcal{D}^{\mathcal{F},h}F(B)$ exists in the sense of (\ref{SMOOTHDER}), $F$ satisfies assumptions~\textbf{(L1,~L2.1,~L2.2)} and~(\ref{i1}) holds for $\mathcal{F}$. Then the differential representation of $F(B)\in \mathcal{W}(\mathbb{F})$ is given by

$$F_t(B_t) = F_0(B_0) + \int_0^t\mathcal{D}F_s(B_s)dB(s) + \int_0^t\mathcal{D}^{\mathcal{F},h}F_s(B_s)ds -\frac{1}{2}\int_0^t\int_{-\infty}^{+\infty}\nabla_xF_s(\textbf{t}(B_s,x))d_{(s,x)} \ell^{x}(s)$$
for $~0\le t\le T.$
\end{theorem}
\begin{proof}
By assumption, $F(B)\in \mathcal{W}(\mathbb{F})$ and~(\ref{i1}) holds. Then, Proposition \ref{ulocaltime} yields

$$F(B) = F_0(B_0) + \int\mathcal{D}F_s(B_s)dB(s) + V$$
where

$$
V(\cdot)=\lim_{k\rightarrow \infty}\Bigg[\int_0^\cdot \mathbb{D}^{\mathcal{F},k,h}F_s(B_s)ds - \frac{1}{2}\int_0^\cdot\int_{-\infty}^{+\infty}\nabla^ {\mathcal{F},v,k}F_s(B_s,x)d_{(s,x)}\mathbb{L}^{k,x}(s)\Bigg]
$$
weakly in $\mathbf{B}^2(\mathbb{F})$. By using Lemma~\ref{fkst} and the existence of $\mathcal{D}^{\mathcal{F},h}F(B)$, we can argue in the same way as in the proof of Theorem 4.3 in \cite{LOS} to conclude that

$$\lim_{k\rightarrow \infty}\int_0^\cdot \mathbb{D}^{\mathcal{F},k,h}F_s(B_s)ds= \int_0^\cdot\mathcal{D}^{\mathcal{F},h}F_t(B_t)dt$$
weakly in $\textbf{B}^1(\mathbb{F})$. Then $Q(\cdot) := \lim_{k\rightarrow \infty}\int_0^\cdot\int_{-\infty}^{+\infty}\nabla^{\mathcal{F},k,v}F_s(B_s,x)d_{(s,x)}\mathbb{L}^{k,x}(s)$ weakly in $\textbf{B}^1(\mathbb{F})$. We claim that

$$
Q(t) = \int_0^t\int_{-\infty}^{+\infty}\nabla_xF_s(\textbf{t}(B_s,x))d_{(s,x)}\ell^x(s)~a.s;~0\le t < \infty.
$$
Let us consider the stopping time
$$
R_m:= \inf\{0\le t < \infty; |B(t)| \ge 2^ m\};~ m\ge 1.
$$
One should notice that $R_m$ is an $\mathbb{F}^k$-stopping time for every $k\ge 1$. Then, by the very definition
$Q(\cdot\wedge R_m) =\lim_{k\rightarrow \infty}\int_0^\cdot\int_{-2^m}^{2^m}\nabla^{\mathcal{F},v,k}F_s(B_s,x)d_{(s,x)}\mathbb{L}^{k,x}(s\wedge R_m)$ weakly in $\textbf{B}^1(\mathbb{F})$. For a given arbitrary $\mathbb{F}$-stopping time $J$ (bounded or not), let $J^\ell$ be a sequence of $\mathbb{F}$-stopping times from Lemma~\ref{fkst} such that $J^\ell$ is an $\mathbb{F}^\ell$-stopping time for each $\ell \ge 1$ and $\lim_{\ell\rightarrow \infty}J^\ell = J~a.s$. Initially, we  set $k > \ell >m$ and let us denote $A^{k,m}$ by the stopped process $A^k$ at $R_m$.

By taking the structure $\mathcal{F} = \big((\mathbf{F}^k)_{k\ge 1},\mathscr{D}\big)$ for $F(B)$, we readily see that $\frac{\Delta F_{\cdot,j,k}}{2^{-k}}(b^k_\cdot(j) - b^k_\cdot(j-1))$ is $\mathbb{F}^k$-predictable for each $j\in \mathbb{Z}$~(recall the notation in~(\ref{deltaF})). Then, we shall use $\mathbb{F}^k$-dual predictable projection on the stochastic set $]]0,J^\ell]]$ (see Th. 5.26 in~\cite{he}) to get

$$\mathbb{E}\int_0^{J^\ell} \frac{\Delta F_{s,j,k}}{2^{-k}}\big(b^k_s(j) - b_s^k(j-1)\big)d[A^{k,m},A^{k,m}](s) =\mathbb{E}\int_0^{J^\ell} \frac{\Delta F_{s,j,k}}{2^{-k}}\big(b^k_s(j) - b_s^k(j-1)\big)d\langle A^{k,m},A^{k.m}\rangle(s);$$
for $j\in \mathbb{Z}$, so that
\small
\begin{equation}\label{prconv}
\mathbb{E}\int_0^{J^\ell}\int_{-2^m}^{2^m} \nabla^{\mathcal{F},v,k}F_s(B_s,x)d_{(s,x)}L^{k,x}(s\wedge R_m)  = \mathbb{E}\int_0^{J^\ell}\int_{-2^m}^{2^m} \nabla^{\mathcal{F},v,k}F_s(B_s,x)d_{(s,x)}\mathbb{L}^{k,x}(s\wedge R_m).
\end{equation}
\normalsize
The fact that $R_m < \infty$ a.s, Proposition~\ref{ltlema} and (\ref{prconv}) yield

\small
\begin{equation}\label{part1}
\mathbb{E}\int_0^{J^\ell}\int_{-2^m}^{2^m} \nabla_x F_s(\textbf{t}(B_s,x))d_{(s,x)}\ell^{x}(s\wedge R_m)= \mathbb{E}\big[Q(J^\ell\wedge R_m)\big];~m\ge 1,
\end{equation}
\normalsize
for every $\ell\ge 1$. Since $Q\in \textbf{B}^1(\mathbb{F})$ has continuous paths, we actually have

\begin{equation}\label{part2}
\lim_{\ell\rightarrow \infty}\mathbb{E}\big[Q(J^\ell \wedge R_m)\big] =\mathbb{E}\big[Q(J\wedge R_m)\big],
\end{equation}
for every $\mathbb{F}$-stopping time $J$. We also know that the 2D Young integral $\int_0^\cdot\int_{-2^m}^{2^m}\nabla_xF_s(\textbf{t}(B_s,x))d_{(s,x)}\ell^{x}(s\wedge R_m)$ has continuous paths for every $m\ge 1$. Moreover, by applying the inequality~given in~[\cite{OS};Corollary 1.1] together with~\textbf{(L2.1)}, we readily see that it belongs to $\textbf{B}^1(\mathbb{F})$. The bounded convergence theorem and relations~(\ref{part1}) and~(\ref{part2}) yield

\begin{equation}\label{part3}
\mathbb{E}\big[Q(J\wedge R_m)\big] = \mathbb{E}\int_0^{J\wedge R_m}\int_{-2^m}^{2^m}\nabla_xF_s(\textbf{t}(B_s,x))d_{(s,x)}\ell^{x}(s),
\end{equation}
for every $m\ge 1$. We shall write, $\mathbb{E}\big[Q(J \wedge R_m)\big] = \mathbb{E}\big[Q(R_m)1\!\!1_{\{J=+\infty\}}\big]  + \mathbb{E}\big[Q(J\wedge R_m)1\!\!1_{\{J< +\infty\}}\big]$. The path continuity of $Q\in \textbf{B}^{1}(\mathbb{F})$ and the fact that all processes are assumed to be null after time $0 < T< \infty$ yield

\begin{equation}\label{gpart4}
\lim_{m\rightarrow  \infty}\mathbb{E}\big[Q(J \wedge R_m)\big] = \mathbb{E}\big[Q(J)1\!\!1_{\{J< +\infty\}}\big].
\end{equation}
By using the fact that the Brownian local time has compact support, we do the same argument to get

\begin{equation}\label{gpart5}
\lim_{m\rightarrow \infty} \mathbb{E}\int_0^{J\wedge R_m}\int_{-2^m}^{2^m}\nabla_xF_s(\textbf{t}(B_s,x))d_{(s,x)}\ell^{x}(s) =\mathbb{E}\int_0^{J}\int_{-\infty}^{+\infty}\nabla_xF_s(\textbf{t}(B_s,x))d_{(s,x)}\ell^{x}(s) 1\!\!1_{\{J< +\infty\}}
\end{equation}
Summing up~(\ref{part3}), (\ref{gpart4}) and (\ref{gpart5}), we conclude that

$$
\mathbb{E}\big[Q(J)1\!\!1_{\{J < +\infty\}}\big] = \mathbb{E}\int_0^J\int_{-\infty}^{+\infty}\nabla_xF_s(\textbf{t}(B_s,x))d_{(s,x)}\ell^{x}(s)1\!\!1_{\{J < +\infty\}}.
$$
Lastly, from Corollary~4.13 in~\cite{he}, we shall conclude that both $Q(\cdot)$ and $\int_0^\cdot\int_{-\infty}^{+\infty}\nabla_xF_s(\textbf{t}(B_s,x))d_{(s,x)}\ell^{x}(s)$ are indistinguishable.
\end{proof}

Let us come back to the Example \ref{exampleYOUNG} under the regularity condition $\frac{1}{2}<\gamma_1 \le 1$.

\begin{example}\label{lastexample}
\end{example}
\noindent In the Example \ref{exampleYOUNG}, let us assume that (\ref{Holderpr}) holds with the additional regularity condition $\frac{1}{2}<\gamma_1 \le 1$. Then,

\begin{eqnarray}
\nonumber F_t(B_t) &=& \int_0^t\int_0^s\varphi(B(r),B(s))dr dB(s) + \int_0^t\int_{-\infty}^{B(s)}\varphi(B(s),y)dyds\\
\label{finalex}& &\\
\nonumber &-& \frac{1}{2}\int_0^t\int_{-\infty}^{+\infty}\int_0^s\varphi(B(r), x)dr d_{(s,x)}\ell^x(s);~0\le  t\le T.
\end{eqnarray}

In view of Theorems \ref{FIRSTTH} and \ref{youngTh}, in order to prove (\ref{finalex}), we only need to check that \textbf{L1, L2.1, L2.2} hold true.
By the very definition, $F\in\mathcal{C}^1$. Let us begin by checking assumption (\textbf{L1}). Notice that $\partial_x F_t(\textbf{t}(c_t,x)) = \int_0^t \varphi(c(s),x)ds$.
Therefore, for c\`adl\`ag paths $c$ and $d$, we easily get

$$
|\nabla_x F_t(\textbf{t}(c_t,x)) - \nabla_x F_t(\textbf{t}(d_t,x))|\le M_2T \sup_{0\le t\le T} |c(t)-d(t)|^{\gamma_2},
$$
which clearly implies assumption (\textbf{L1}). Let us now check assumption \textbf{L2.1} and \textbf{L2.2}. H\"{o}lder continuity yields

$$
|\Delta_i\Delta_j\nabla_xF_{t_i}(\textbf{t}(B_{t_i},x_j))|\le  M_1 (t_i-t_{i-1})(x_j-x_{j-1})^{\gamma_1}.
$$

By making the change of variables $z = y-\Delta j_k(x_j)2^{-k}$ with $\Delta j_k(x_j) = (j_k(x_{j})-j_k(x_{j-1}))$, we similarly have

$$
|\Delta_i\Delta_j\nabla^{\mathcal{F},v,k} F_{t_i}(B_{t_i},x_j)| \le M_1(t_i-t_{i-1})(x_j-x_{j-1})^{\gamma_1}.
$$
For every $L>0$, we clearly have
$$\sup_{0\le t\le T}\| \nabla_xF_t(\textbf{t}(B_t,\cdot))\|_{[-L,L];p}\le M_12L~a.s$$
where $p=\frac{1}{\gamma_1}$. Similarly, the usual mean value theorem yields
$$\sup_{k\ge 1}\sup_{0\le t\le T}\| \nabla^{\mathcal{F},k,v} F_t(B_t,\cdot))\|_{[-L,L];p}\le M_12L~a.s.$$
Now, if $1/2 < \gamma_1 \le1$, there exists $\delta >0$ such that $\gamma_1\in (\frac{1+\delta}{2+\delta}, 1]$ and we shall take $q_1=1,~q_2 = p$ and $\alpha\in (0,1)$ in such way that $\min\{\alpha + \frac{1}{q_1},\frac{1-\alpha}{2+\delta} +\frac{1}{q_2}\}>1$ and $\frac{1}{p} + \frac{1}{2+\delta}> 1$. This shows that \textbf{L2.1} and \textbf{L2.2} hold true. It remains to show that

$$\mathcal{D}^{\mathcal{F},h}X(t) = \int_{-\infty}^{B(t)}\varphi(B(t),y)dy;0\let\le T.$$
But this is a simple consequence of the identity

$$\mathcal{D}^{\mathcal{F},k,h}F_{T^k_n}(B_{T^k_n}) =\int_{-\infty}^{A^k(T^k_{n-1})}\varphi(A^k(T^k_{n-1}),y)2^{2k}(T^k_n-T^k_{n-1})dy; n\ge 1$$
and the fact that $\mathbb{E}|T^k_n-T^k_{n-1}|=2^{-2k}$ where $T^k_n-T^k_{n-1}$ is independent from $\int_{-\infty}^{A^k(T^k_{n-1})}\varphi(A^k(T^k_{n-1}),y)dy$ for every $n\ge 1$.

\end{document}